\documentclass[12pt]{amsart}
\usepackage{tikz-cd}
\usepackage{enumitem}
\setlist[description]{leftmargin=\parindent,labelindent=\parindent}
\usepackage{amssymb}
\usepackage{amsmath,amscd}
\usepackage{mathrsfs}
\usepackage{url}
\usepackage{mathtools}
\usepackage{cite}
\usepackage{fullpage}
\usepackage{hyperref}
\usepackage{verbatim}
\usepackage{comment}
\usepackage{fancyvrb}
\usepackage{fvextra}
\usepackage{caption}
\usepackage{tikz}
\usepackage{tkz-graph}
\usetikzlibrary{matrix}
 \usetikzlibrary{shapes}
\usetikzlibrary{arrows,decorations.markings}
\usepackage{tikz-cd}
\usepackage{stmaryrd}

\pgfarrowsdeclare{bad to}{bad to}
{
  \pgfarrowsleftextend{-2\pgflinewidth}
  \pgfarrowsrightextend{\pgflinewidth}
}
{
  \pgfsetlinewidth{0.8\pgflinewidth}
  \pgfsetdash{}{0pt}
  \pgfsetroundcap
  \pgfsetroundjoin
  \pgfpathmoveto{\pgfpoint{-3\pgflinewidth}{4\pgflinewidth}}
  \pgfpathcurveto
  {\pgfpoint{-2.75\pgflinewidth}{2.5\pgflinewidth}}
  {\pgfpoint{0pt}{0.25\pgflinewidth}}
  {\pgfpoint{0.75\pgflinewidth}{0pt}}
  \pgfpathcurveto
  {\pgfpoint{0pt}{-0.25\pgflinewidth}}
  {\pgfpoint{-2.75\pgflinewidth}{-2.5\pgflinewidth}}
  {\pgfpoint{-3\pgflinewidth}{-4\pgflinewidth}}
  \pgfusepathqstroke
}

\newtheorem{thm}{Theorem}[section]
\newtheorem{prop}[thm]{Proposition}
\newtheorem{lem}[thm]{Lemma}
\newtheorem{cor}[thm]{Corollary}

\theoremstyle{definition}

\newtheorem{rem}[thm]{Remark}

\numberwithin{equation}{section}

\renewcommand{\L}{\mathcal{L}}

\newcommand{\X}{\mathcal{X}}

\newcommand{\V}{\mathcal{V}}
\newcommand{\J}{\mathcal{J}}
\newcommand{\N}{\mathcal{N}}
\newcommand{\B}{\mathcal{B}}
\newcommand{\Z}{\mathcal{Z}}

\newcommand{\zz}{\mathbb{Z}}

\newcommand{\qq}{\mathbb{Q}}

\newcommand{\C}{\mathcal{C}}
\newcommand{\W}{\mathcal{W}}

\newcommand{\pp}{\mathbb{P}}
\renewcommand{\H}{\mathcal{H}}

\newcommand{\F}{\mathcal{F}}
\renewcommand{\P}{\mathcal{P}}

\newcommand{\E}{\mathcal{E}}
\renewcommand{\O}{\mathcal{O}}

\renewcommand{\tilde}{\widetilde}

\DeclareMathOperator{\SL}{SL}
\DeclareMathOperator{\Pic}{Pic}

\DeclareMathOperator{\Sym}{Sym}

\DeclareMathOperator{\PGL}{PGL}

\DeclareMathOperator{\BSL}{BSL}
\DeclareMathOperator{\BPGL}{BPGL}
\DeclareMathOperator{\BGL}{BGL}

\renewcommand{\gg}{\mathbb{G}}

\newcommand{\sV}{\mathscr{V}}
\newcommand{\sB}{\mathscr{B}}
\newcommand{\sP}{\mathscr{P}}
\newcommand{\sJ}{\mathscr{J}}
\newcommand{\sM}{\mathscr{M}}
\newcommand{\sC}{\mathscr{C}}
\newcommand{\sL}{\mathscr{L}}
\newcommand{\sH}{\mathscr{H}}

\usepackage{wrapfig}

\title{The Chow ring of the universal Picard stack over the hyperelliptic locus}
\author{Hannah Larson}

\begin{document}

\maketitle

\begin{abstract}
Let $\sJ^d_g \to \sM_g$ be the universal Picard stack parametrizing degree $d$ line bundles on genus $g$ curves, and let $\sJ^d_{2,g}$ be its restriction to locus of hyperelliptic curves $\sH_{2,g} \subset \sM_g$. We determine the rational Chow ring of $\sJ^d_{2,g}$ for all $d$ and $g$. In particular, we prove it is generated by restrictions of tautological classes on $\sJ^d_g$ and we determine all relations among the restrictions of such classes. We also compute the integral Picard group of $\sJ^d_{2,g}$, completing (and extending to the $\PGL_2$-equivariant case) prior work of Erman and Wood.
As a corollary, we prove that $\sJ^d_{2,g}$ is either a trivial $\gg_m$-gerbe over its rigidification, or has Brauer class of order $2$, depending on the parity of $d - g$.
\end{abstract}

\section{Introduction}
Understanding the line bundles on curves is essential to understanding the curves themselves. As such, the universal Picard stack
plays a prominent role in the intersection theory of moduli of curves, recently for example in connection with the double ramification cycle \cite{DR}.
Like the moduli space of curves, the universal Picard stack 
has a distinguished collection of classes in its Chow ring called \emph{tautological classes}.
We briefly review their construction, which generalizes that of kappa classes on the moduli space of curves.

Let $\sM_g$ be the moduli space of genus $g$ curves, and let $\sJ^d_g \to \sM_g$ be the universal Picard stack parameterizing degree $d$ line bundles on genus $g$ curves.
We write
$f:\sC \to \sJ^d_g$ for the pullback of the universal curve from $\sM_g$. Then there are two natural line bundles on $\sC$: the relative dualizing sheaf $\omega_f$, and the universal degree $d$ line bundle, which we call $\sL$. From these we obtain the \emph{twisted kappa classes}
\[\kappa_{i,j} := f_*(c_1(\omega_f)^{i+1} \cdot c_1(\sL)^j) \in A^{i + j}(\sJ^d_{g}) \qquad \text{for $i \geq -1, j \geq 0$}.\]
Note that $\kappa_{i,0}$ is the pullback of $\kappa_i$ from $\sM_g$.

In analogy with the Madsen--Weiss theorem on the stable cohomology of $\sM_g$ \cite{MadsenWeiss}, it is known that the 
the stable cohomology of $\sJ^0_g$ as $g$ tends to infinity is freely generated by the twisted kappa class \cite{CohenMadsen,EW}. The relations among twisted kappa classes for fixed $g$ were studied in \cite{BL}, which discovered a family of ``Faber--Zagier type" relations.
However, despite much interest in $\sJ^d_g$ and its tautological subring, there are no cases (beyond genus $0$ or $1$) where either the tautological subring or the full Chow ring of $\sJ^d_{g}$ is known.

Over the past several years, the Chow ring of $\sM_g$ has been determined for $g \leq 9$ \cite{FaberI,FaberII,Izadi,PenevVakil,789}. Each of these works sets up the problem considering the stratification of $\sM_g$ by gonality. The first step of such calculations is to understand the smallest stratum, namely the hyperelliptic locus $\sH_{2,g} \subset \sM_g$.  For $\sM_g$, this smallest stratum is particularly nice:
The fundamental class of $\sH_{2,g}$ generates the socle in the tautological ring of $\sM_g$, and, with rational coefficients, there are no other classes supported on it.

One might hope to similarly approach the intersection theory of $\sJ^d_g$ via the gonality stratification. The first step here is to study the restriction of $\sJ^d_g \to \sM_g$ to the hyperelliptic locus, which we call $\sJ^d_{2,g} \to \sH_{2,g}$. Our main theorem determines the rational Chow ring of $\sJ^d_{2,g}$ for all $d$ and $g$. 
In contrast with $\sH_{2,g}$, we find that $\sJ^d_{2,g}$ already supports many interesting classes.  
As a special case, this provides the first non-trivial complete calculation of any $A^*(\sJ^d_g)$ in the case $g = 2$.
Below, by slight abuse of notation, we write $\kappa_{i,j}$ for the restriction of $\kappa_{i,j}$ along the inclusion $\sJ^d_{2,g} \subset \sJ^d_g$.

\begin{thm} \label{intro}
Given an integer $g \geq 2$ and any integer $d$, the rational Chow ring of $\sJ^d_{2,g}$ is
\[A^*(\sJ^d_{2,g}) = \frac{\qq[\kappa_{0,1}, \kappa_{-1,2}]}{\langle (d\kappa_{0,1} - (g-1)\kappa_{-1,2})^{g+1}\rangle}.\]
\end{thm}

We remark that $\sJ^d_{2,g}$ is a $\gg_m$-gerbe over its rigidification $J^d_{2,g}$. This explains why its Chow ring possesses
non-trivial classes in arbitrarily high codimension.
As we shall explain in Section \ref{rigid}, the Chow ring of $J^d_{2,g}$ is a subring living in finitely many degrees:
\[A^*(J^d_{2,g}) = \qq[u]/(u^{g+1}) \hookrightarrow A^*(\sJ^d_{2,g}) \qquad \text{given by} \qquad  u \mapsto d\kappa_{0,1} - (g - 1)\kappa_{-1,2}.\]
This also gives us a distinguished description of the codimension $1$ class whose $(g+1)$st power vanishes.
When $d = g - 1$, this distinguished class is a multiple of the class of the theta divisor (see Section \ref{theta}).

As a corollary, we determine all relations among the restrictions of tautological classes to the hyperelliptic locus. The shape of these relations may be of independent interest since, for example, the ideal they generate must contain the ideal of tautological relations on $\sJ^d_g$.

\begin{cor} \label{relscor}
Let $g \geq 2$ and let $d$ be any integer. The relations among the restrictions of tautological classes to $\sJ^d_{2,g}$  are generated by the following:
\begin{align*} \kappa_{i,j} &= 0 \qquad \text{for all $i \geq 1, j \geq 0$} \\
\kappa_{0,j} &= \frac{1}{(2g - 2)^{j-1}} (\kappa_{0,1})^j \\
\kappa_{-1,j}
&= \frac{1}{(2g-2)^{j-1}} \cdot(\kappa_{0,1})^{j-2}\cdot \left( (g-1)(j^2 - j) \cdot \kappa_{-1,2} - d(j^2 - 2j) \cdot \kappa_{0,1}\right) \\
0 &= (d\kappa_{0,1} - (g - 1)\kappa_{-1,2})^{g+1}.
\end{align*}
\end{cor}

Perhaps most interesting is the codimension $g+1$ relation among
the codimension $1$ generators, which
also appears in the theorem statement. 
Our proof of this relation depends on the parity of $d$, but somewhat magically yields the same formula for all $d$. 
The geometric explanation of this relation --- and indeed the core of our approach more broadly --- relies on
the relationship between line bundles on hyperelliptic curves and rank $2$ vector bundles on $\pp^1$. Namely, 
given a line bundle $L$ on a hyperelliptic curve $\alpha: C \to \pp^1$, 
its pushforward $\alpha_*L$ is a rank $2$ vector bundle on $\pp^1$.
By the Grothendieck--Birkhoff theorem, this vector bundle splits as a direct sum of line bundles $\O_{\pp^1}(e_1) \oplus \O_{\pp^1}(e_2)$ with $e_1 \leq e_2$. We call the pair $(e_1, e_2)$ the \emph{splitting type of $L$}.
Using Riemann--Roch, one sees that 
\begin{equation} \label{rr} e_1 + e_2 + 2 = \chi(\alpha_*L) = \chi(L) = d - g + 1 \qquad \Rightarrow \qquad e_1 + e_2 = d - g - 1.
\end{equation}
In addition, the Brill--Noether theory of hyperelliptic curves tells us that $e_2 \leq \frac{d}{2}$, and equality holds if and only if $L = \alpha^*\O_{\pp^1}(\frac{d}{2})$.

When $d$ is odd, there are no line bundles of splitting type $(\lceil d/2 \rceil, \lfloor d/2 \rfloor - (g + 1))$. The formula for the class of this splitting locus gives rise to the codimension $g+1$ relation. 
Meanwhile, when $d$ is even, the splitting type $(d/2, d/2 - (g + 1))$ plays a special role. The line bundles with this splitting type are exactly those of the form $\alpha^*\O_{\pp^1}(\frac{d}{2})$. 
Our approach to studying $\sJ^d_{2,g}$ involves excising a discriminant locus from a universal linear system of curves on Hirzebruch surfaces.
It turns out that the map from the universal singular point to the universal Hirzebruch surface jumps fiber dimension precisely along the directrices within Hirzebruch surfaces of this most unbalanced splitting type $(d/2, d/2 - (g + 1))$. This gives rise to a codimension $g+1$ relation obtained as the pushforward of a certain codimension $1$ class living inside this codimension $g$ splitting locus.

\subsection{The integral Picard group}
The Picard group of $\sJ^d_g$ was found in \cite{MeloViviani,EW}.
Here, we study the Picard group of $\sJ^d_{2,g}$, completing an analysis begun by Erman and Wood in \cite{ErmanWood}.

Let $\H_{2,g}^{\dagger}$ be the Hurwitz space of 
degree $2$ covers of a fixed $\pp^1$. The map
$\H_{2,g}^{\dagger} \to \sH_{2,g}$ is a $\PGL_2$ bundle.
Erman and Wood considered the universal Picard stack $\J^{d,\dagger}_{2,g}$ over $\H_{2,g}^{\dagger}$ and 
computed its integral Picard group when $d - g$ is even \cite[Theorem 7.2]{ErmanWood}.
Erman and Wood make use of the stratification of $\J^{d,\dagger}_{2,g}$ by the splitting type of line bundles. 
Specifically, they find the Picard group of the largest piece in this stratification.
When $d -g$ is even, the complement of the largest piece in the stratification has codimension $2$, so its Picard group agrees with that of $\J^{d,\dagger}_{2,g}$.
However, when $d - g$ is odd, the complement of the largest piece has codimension $1$, so that studying piece alone is not enough to determine the Picard group.

The novel contribution in this paper is to leverage the Chow ring of the moduli stack of vector bundles on $\pp^1$ to ``put together" the intersection theory of all the pieces in this stratification. In addition to the main theorem,
this allows us to complete the missing case from \cite{ErmanWood} when $d - g$ is odd. The answer is the same as when $d - g$ is even.

\begin{thm} \label{pthm}
For $g \geq 2$ and any $d$, we have
\[\Pic(\J^{d,\dagger}_{2,g}) = \zz \oplus \zz \oplus \zz/(8g + 4).\]
\end{thm}

We also extend the calculation of the integral Picard group to the $\SL_2$ and $\PGL_2$ equivariant settings. Let us write $\J^d_{2,g} = [\J^{d,\dagger}_{2,g}/\SL_2]$ and $\sJ^d_{2,g} = [\J^{d,\dagger}_{2,g}/\PGL_2]$ for the respective quotient stacks. We note that $\J^d_{2,g} \to \sJ^d_{2,g}$ is a $\mu_2$-gerbe. Thus, their rational Chow rings are isomorphic, and in fact, we shall work with $\J^d_{2,g}$ for the rational computations. However, their integral Picard groups are not always the same.
We will see that $\Pic(\J^d_{2,g}) = \Pic(\J^{d,\dagger}_{2,g})$, but when working $\PGL_2$ equivariantly, the answer turns out to depend on the parity of $g$.

\begin{thm} \label{picthm}
For $g \geq 2$ and any $d$, we have
\[\Pic(\sJ^{d}_{2,g}) =\begin{cases} \zz \oplus \zz \oplus \zz/(8g + 4) & \text{if $g$ is odd} \\ 
\zz \oplus \zz \oplus \zz/(4g + 2) & \text{if $g$ is even.}
\end{cases}
\]
\end{thm}

The integral Picard group of $\sH_{2,g}$  was computed in \cite[Theorem 5.1]{ArsieVistoli} and it also depends on the parity of $g$. 
Combining that result with our proof of Theorem \ref{picthm}, it follows that the pullback map $\Pic(\sH_{2,g}) \to \Pic(\sJ^d_{2,g})$ is exactly the inclusion of the torsion subgroup. Explicit line bundles on $\mathscr{J}^d_{2,g}$ generating the Picard group are given in Section \ref{clb}.

It was shown in \cite[Theorem 6.4]{MeloViviani} that the Brauer class of the $\gg_m$-gerbe $\mathscr{J}^d_g \to J^d_g$ has order $\gcd(d - g + 1, 2g - 2)$. Using Theorem \ref{picthm} above, we obtain the following result for its restriction to the hyperelliptic locus.

\begin{cor} \label{bc}
The Brauer class of the $\gg_m$-gerbe $\mathscr{J}^d_{2,g} \to J^d_{2,g}$ has order $\gcd(d - g + 1, 2)$.
\end{cor}
Tensoring with the $g^1_2$ shows that all $\sJ^d_{2,g}$ with $d$ of the same parity are isomorphic. The above result shows that those of opposite parity cannot be isomorphic as $\gg_m$-stacks.

\begin{rem}
(1) Corollary \ref{bc} provides a new proof of \cite[Corollary 3.1]{MR}, which showed that a Poincar\'e line bundle exists on $J^d_{2,g}$ if and only if $d$ and $g$ have the same parity. Conversely, using that result and arguing as in \cite[Proposition 6.6]{MeloViviani} gives rise to an alternative proof of Corollary \ref{bc}.

(2) The Picard stack of a curve over a non-algebraically closed field is typically a non-trivial gerbe over its Picard scheme.
For genus $1$ curves, this Brauer class was studied and related to the period-index problem in \cite{AG}.
\end{rem}

\subsection{Overview of Techniques}
They key to understanding $\sJ^d_{2,g}$ is its relationship with a moduli stack $\sB^d_{2,g}$ of certain pairs of vector bundles on $\pp^1$. In particular, $\sJ^d_{2,g}$ is an open substack of a vector bundle $\mathscr{X}$ over $\sB^d_{2,g}$. Thus, the Chow ring of $\sJ^d_{2,g}$ is a quotient of the Chow ring of $\sB^d_{2,g}$.
Section \ref{s2} is devoted to defining
$\sB^d_{2,g}$ and determining its Chow ring. This builds off of our previous calculation of the Chow ring of the moduli stack of vector bundles on $\pp^1$ in \cite{Brd}; however, it is not an immediate consequence of earlier work, as we must restrict to an open substack where the splitting type is bounded, which introduces additional relations in the Chow ring. The relations all occur in codimension $g+1$ and higher, and proving that our claimed relations on $\sB^d_{2,g}$ are complete involves a careful analysis with universal splitting loci.

The next step is to study the closed complement of $\sJ^d_{2,g}$ inside $\mathscr{X}$. The complement can be described as the discriminant locus of a linear series of curves on the universal Hirzebruch surface over $\sB^d_{2,g}$. We control the relations that arise from excising this discriminant locus by using relative bundles of principal parts. We briefly review the construction and basic properties of relative bundles of principal parts in Section \ref{partsec}. We then use them  in Section \ref{ex} to construct the universal singular point, whose Chow groups surject onto the Chow groups of the discriminant locus. Even though the arguments in this section depend significantly on the parity of $d$, the final result for the rational Chow ring is the same. We also show that our codimension $1$ calculations hold integrally, proving Theorems \ref{pthm}  and \ref{picthm} in Section \ref{picsec}.

Finally, in Section \ref{taut}, we take advantage of an embedding of the universal curve over $\sJ^d_{2,g}$ inside a Hirzebruch surface to compute the twisted kappa classes in terms of our generators that were pulled back from $\sB^d_{2,g}$. We can then translate the presentation of $A^*(\sJ^d_{2,g})$ given in Section \ref{ex} into the form presented in Theorem \ref{intro}. To express the higher codimension kappa classes in terms of our codimension $1$ generators, we set up a system of recursive relations. Solving this recursion proves Corollary \ref{relscor}.

\begin{rem}
Let us also mention some recent related work. 
When we stratify $\sJ^d_{2,g}$ according to the splitting type of line bundles, each stratum is the moduli space of hyperelliptic curves on some fixed Hirzebruch surface.
The stable cohomology of each of these strata was found in \cite{BergZ}.
\end{rem}

\subsection{Acknowledgements} This research was conducted during the period I served as a Clay Research Fellow. I would like to thank Younghan Bae, Samir Canning, and Martin Olsson for valuable conversations related to this work.
 
\section{Construction of the moduli stack} \label{s2}

In \cite{ErmanWood}, Erman and Wood relate $\J^{d,\dagger}_{2,g}$ to a moduli stack of vector bundles on $\pp^1$. While our presentation is slightly different, the key ideas that we present below and in Sections \ref{themor} and \ref{stackdef} are due to them. The important original contributions of this section are in Sections \ref{usl} and \ref{cb}, where we determine the Chow ring of the stack $\B^d_{2,g}$ defined below.

Suppose we are given a hyperelliptic curve $\alpha: C \to \pp^1$ and a degree $d$ line bundle $L$ on $C$. There are two naturally associated vector bundles on $\pp^1$. First, there is the rank $2$ vector bundle $E := \alpha_*L$, which by \eqref{rr} has degree $d - g - 1$. If we denote the splitting type of $E$ by $(e_1, e_2)$ with $e_1 \leq e_2$, then 
\[h^0(C, L \otimes \alpha^*\O_{\pp^1}(-e_2)) = h^0(\pp^1, E(-e_2)) \geq 1.\]
Hence, the degree of $L \otimes \alpha^*\O_{\pp^1}(-e_2)$ must be nonnegative. This implies
$d - 2e_2 \geq 0$, or equivalently, $e_1 \geq \frac{d}{2} - g - 1$. 

In the special case that $L = \O_C$, the pushforward $\alpha_*\O_C$ necessarily has an $\O_{\pp^1}$ summand. The dual of the quotient is called the \emph{Tschernhausen bundle} of the cover, which we'll denote $F := (\alpha_*\O_C/\O_{\pp^1})^\vee$. By \eqref{rr}, the line bundle $F$ has degree $g +1$.
 As we will explain below, the association of $(\alpha: C \to \pp^1, L)$ with the pair of vector bundles $(E, F)$ defines a morphism of $\J^{d,\dagger}_{2,g}$ to a moduli stack $\B^{d,\dagger}_{2,g}$ parameterizing pairs of vector bundles on $\pp^1$. We would also like to allow the $\pp^1$ target to vary in families, in other words take the quotient of both of these stacks by $\SL_2$ or $\PGL_2$. Below, we use caligraphic font for the $\SL_2$ quotient and script font for the $\PGL_2$ quotient. The caligraphic version is always a $\mu_2$-gerbe over the script version.
 The distinction between these will only matter for the computation of the integral Picard group.

 \subsection{The morphism $\J^{d,\dagger}_{2,g} \to \B^{d,\dagger}_{2,g}$} \label{themor}
An object of the stack $\J^{d,\dagger}_{2,g}$ over a scheme $S$ is the data of a pair $(\alpha: C \to S \times \pp^1, L)$ 
where
\begin{enumerate}
    \item $\alpha: C \to S \times \pp^1$ is a degree $2$ cover such that the composition $C \to S \times \pp^1 \to S$ is a family of smooth genus $g$ curves;
    \item  $L$ is a line bundle on $C$ having degree $d$ on fibers of $C \to S$.
\end{enumerate}
Given $\alpha$ and $L$ as above, the pushforward $E := \alpha_*L$ is a rank $2$ vector bundle on $S \times \pp^1$ having degree $d - g - 1$ on fibers over $S$. In addition, $F := (\alpha_*\O_C/\O_P)^\vee$ is a line bundle of relative degree $g+1$ on $S \times \pp^1$.

Let us define $\B^{d, \dagger}_{2,g}$ to be the stack whose objects over a scheme $S$ are pairs $(E, F)$ where 
\begin{enumerate}
\item $E$ is a rank $2$, relative degree $d -g - 1$ vector bundle on $S \times \pp^1$ whose restrictions to fibers of $S \times \pp^1 \to S$ have splitting type $(e_1, e_2)$ with $e_1\geq \frac{d}{2} - g - 1$;
\item $F$ is a relative degree $g+1$ line bundle on $S \times \pp^1$. 
\end{enumerate}
Sending $(\alpha: C \to S \times \pp^1)$ to $(E, F)$ defines a morphism of stacks $\J^{d,\dagger}_{2,g} \to \B^{d,\dagger}_{2,g}$.
We define $\B^d_{2,g} = [\B^{d,\dagger}_{2,g}/\SL_2]$ and $\sB^d_{2,g} = [\B^{d,\dagger}_{2,g}/\PGL_2]$.
Taking the $\SL_2$ or $\PGL_2$ quotient of our morphism above, we obtain morphisms $\J^d_{2,g} \to \B^d_{2,g}$ and $\sJ^d_{2,g} \to \sB^d_{2,g}$. These are defined in the same way, where $S \times \pp^1$ is replaced with a $\pp^1$-bundle $\pp V$ where $V$ is a rank $2$ vector bundle on $S$ with trivial determinant in the $\SL_2$ case, or it is replaced with a family $P \to S$ of smooth genus $0$ curves in the $\PGL_2$ case.

An essential fact, which we shall show in Section \ref{stackdef}, is that the morphism $\J^d_{2,g} \to \B^d_{2,g}$ factors as an open inclusion inside of a vector bundle over $\B^{d}_{2,g}$, and similarly with script letters.
Consequently, the Chow ring of $\J^d_{2,g}$ is generated by the pullbacks of classes from $\B^d_{2,g}$. 
The main focus of this section is to understand the intersection theory of $\B^d_{2,g}$.

\subsection{Stacks of vector bundles on $\pp^1$}
Let $\V_{r,d}^{\dagger}$ denote the moduli stack of rank $r$, degree $d$ vector bundles on $\pp^1$. An object of $\V_{r,d}^{\dagger}$ over a scheme $S$ is the data of a rank $r$, relative degree $d$ vector bundle $E$ on $S \times \pp^1$.
A morphism between objects $(S, E)$ and $(S', E')$ is a cartesian diagram
\begin{center}
\begin{tikzcd}
\pp^1 \times S'\arrow{d} \arrow{r}{\xi} & \pp^1 \times S \arrow{d}  \\
S' \arrow{r} & S,
\end{tikzcd}
\end{center}
together with an isomorphism $\phi: \xi^*E \to E'$.

We also require the quotient stacks $\V_{r,d} = [\V_{r,d}^{\dagger}/\SL_2]$ and $\sV_{r,d} = [\V_{r,d}^{\dagger}/\PGL_2]$. Explicitly, the objects of $\V_{r,d}$ over a scheme $S$ are triples $(S, V, E)$ where  $V$ is a rank $2$ vector bundle on $S$ with trivial determinant and $E$ is a rank $r$, relative degree $d$ vector bundle on $\pp V$. Meanwhile, the objects of $\sV_{r,d}$ over a scheme $S$ are 
$(S, P, E)$ where $P \to S$ is a family of smooth genus $0$ curves and $E$ is a rank $r$, relative degree $d$ vector bundle on $P$. (The difference between the objects of $\V_{r,d}$ and $\sV_{r,d}$ is that $P \to S$ need not have a relative degree $1$ line bundle.) The natural map $\V_{r,d} \to \sV_{r,d}$ is a $\mu_2$-gerbe. For more background on the moduli stack of vector bundles on $\pp^1$, we refer the reader to \cite{Brd}.

We define
\begin{equation} \label{1} \widehat{\B}^d_{2,g} := \V_{2,d - g - 1} \times_{\BSL_2} \V_{1,g+1} 
\end{equation}
and
\begin{equation} \label{2} \widehat{\sB}^d_{2,g} := \sV_{2,d-g-1} \times_{\BPGL_2} \sV_{1,g+1}.
\end{equation}
These stacks parameterize pairs of vector bundles on $\pp^1$-bundles, respectively familes of genus $0$ curves.
Note that, by construction, $\B^d_{2,g}$ is the open substack of $\widehat{\B}^d_{2,g}$ where the restriction of the universal rank $2$ vector bundle to fibers of the universal genus $0$ curve has splitting type $(e_1, e_2)$ with $e_1 \geq \frac{d}{2} - g - 1$, and similarly for script letters. 

\begin{rem} \label{bgm}
We note that $\V_{1,e}^{\dagger}$ is equivalent to $\mathrm{B}\gg_m$. Suppose $E$ is a line bundle of degree $e$ on $S \times \pp^1$. Then $E(-e)$ is a line bundle of relative degree $0$. Writing $\pi: S \times \pp^1 \to S$ for the projection, cohomology and base change shows there is an isomorphism $E(-e) \cong \pi^*\pi_*E(-e)$. Thus the equivalence $\V_{1,e}^{\dagger} \to \mathrm{B}\gg_m$ is given by sending $(S, E)$ to the line bundle $\pi_*E(-e)$. The inverse map is defined by sending a line bundle $L$ on $S$ to the line bundle $\pi^*L \otimes \O_{\pp^1}(e)$ on $S \times \pp^1$.  A similar argument shows that $\V_{1,e}$ is equivalent to $\mathrm{B}\gg_m \times \BSL_2$. (Replace $S \times \pp^1$ with $\pp V$ and replace $\O_{\pp^1}(1)$ with $\O_{\pp V}(1)$.)

However, the $\PGL_2$ quotient is sensitive to the parity of $e$. Suppose $E$ is a line bundle of degree $e$ on a family of genus $0$ curves $P \to S$. If $e$ is even, then $E\otimes \omega_{P/S}^{\otimes e/2} \cong \pi^*\pi_* (E\otimes \omega_{P/S}^{\otimes e/2})$ and a similar argument applies to show $\sV_{1,e}$ is equivalent to $\mathrm{B}\gg_m \times \BPGL_2$. However, if $e$ is odd, it turns out $\sV_{1,e}$ is equivalent to $\BGL_2$. Given $E$ on $\pi: P \to S$, we produce a rank $2$ vector bundle on $S$ by $V := \pi_*(E \otimes \omega_{P/S}^{\otimes \lfloor e/2 \rfloor})$. Using cohomology and base change, we find that $P \cong \pp V$ and $\O_{\pp V}(1) \cong E \otimes \omega_{P/S}^{\otimes \lfloor e/2 \rfloor}$, so $E$ can be recovered as $\O_{\pp V}(1) \otimes \omega_{P/S}^{-\otimes \lfloor e/2 \rfloor}$.
\end{rem}

We determined the integral Picard groups of stacks of the form in \eqref{1} and \eqref{2} in \cite[Lemma 3.5]{Pic}. To state the result, we need some more notation. We shall follow the notation in \cite{Pic} with one difference: We have not imposed the condition that our vector bundles $E$ and $F$ on $P \to S$ are globally generated on fibers. A ``boundedness condition" does come into play when we restrict to the open substack $\B^d_{2,g} \subset \widehat{\B}^d_{2,g}$, but we first need to understand $\widehat{\B}^d_{2,g}$.

Let us write $\pi: \P \to \widehat{\B}^d_{2,g}$ for the universal $\pp^1$-bundle, pulled back from $\widehat{\B}^d_{2,g} \to \BSL_2$. 
It comes equipped with a relative degree $1$ line bundle $\O_{\P}(1)$. Let $z = c_1(\O_{\P}(1)) \in A^1(\P)$. In addition, let $c_2 = c_2(\pi_*\O_{\P}(1))$ be the pullback of the universal second Chern class from $\BSL_2$. By the projective bundle formula, we have $z^2 = -\pi^*c_2 \in A^2(\P)$.

Next, let $\E$ be the universal rank $2$ vector bundle on $\P$ and let $\F$ be the universal rank $1$ vector bundle on $\P$.
Since $\P \to \widehat{\B}^d_{2,g}$ is a $\pp^1$-bundle, there exist $a_i, a_i', b_i, b_i' \in A^*(\widehat{\B}^d_{2,g})$ such that
\[c_i(\E) = \pi^*a_i + \pi^* a_i' z \qquad \text{and} \qquad 
c_i(\F) = \pi^*b_i + \pi^*b_i'z.\]
Note that $a_1' = d - g - 1$ and $b_1' = g+ 1$ are the relative degrees of $\E$ and $\F$ respectively. Since $\pi^*: A^*(\widehat{\B}^d_{2,g}) \to A^*(\P)$ is an inclusion, in what follows, we shall often omit $\pi^*$'s from notation simply identifying $a_i, a_i', b_i, b_i'$ and $c_2$ with their images under it.

By the sentence before \cite[Lemma 3.5]{Pic},  $\Pic(\widehat{\B}^d_{2,g})$ is freely generated by $a_1, a_2'$ and $b_1$. The following is a consequence of \cite[Lemma 3.5]{Pic}. (Note that $b_2' = 0$ since our $\F$ has rank $1$.)
\begin{lem} \label{pin}
The natural map $\widehat{\B}^d_{2,g} \to \widehat{\sB}^d_{2,g}$ induces an inclusion
\[\Pic(\widehat{\sB}^d_{2,g}) \hookrightarrow \Pic(\widehat{\B}^d_{2,g}) = \zz a_1 \oplus \zz a_2' \oplus \zz b_1\]
whose image is the subgroup generated by
\[\begin{cases}
a_1, a_2', b_1 & \text{if $g$ is odd and $d$ is even} \\
2a_1, a_2', b_1 & \text{if $g$ and $d$ are both odd} \\
a_1, a_2', 2b_1 & \text{if $g$ is even and $d$ is odd} \\
a_1 - b_1, 2a_1, a_2' &\text{if $g$ and $d$ are both even.}
\end{cases}\]
\end{lem}


Next we show that the rational Chow ring of $\widehat{\B}^d_{2,g}$ is freely generated by $a_1, a_2', a_2,b_1, $ and $c_2$. This is a special case of the following more general fact regarding such fiber products.

\begin{lem} \label{free}
For any $r, d, s, e$, the rational Chow ring of $\V_{r,d} \times_{\BSL_2} \V_{s,e}$ is free on the $a_i, a_i', b_i, b_i'$ and $c_2$. In particular, 
\[A^*(\widehat{\B}^d_{2,g}) = \qq[a_1, a_2, a_2', b_1, c_2].\]
\end{lem}
\begin{proof}
Let $\V^{\mathrm{gg}}_{r,d} \subset \V_{r,d}$ denote the open substack parameterizing globally generated vector bundles.
In \cite[Theorem 4.4]{part1}, we proved that the Chow ring of any product of the form $\V^{\mathrm{gg}}_{r,d} \times_{\BSL_2} \V^{\mathrm{gg}}_{s,e}$ is generated by $a_i, a_i', b_i, b_i',$ and $c_2$, and all relations among these classes lie in codimension at least $\min(d, e) + 1$. 

Note that
twisting both of the universal vector bundles by $\O_{\pp^1}(m)$ gives an isomorphism of $\V_{r,d} \times_{\BSL_2} \V_{s,e}$
with
$\V_{r,d+mr} \times_{\BSL_2} \V_{s,e+ms}$. 
As we increase $m$, the degrees of the vector bundles increase, and the codimension of the complement of 
\[\V^{\mathrm{gg}}_{r,d+mr} \times_{\BSL_2} \V^{\mathrm{gg}}_{s,e+ms} \subset 
\V_{r,d+mr} \times_{\BSL_2} \V_{s,e+ms}
\]
tends to infinity as $m$ tends to infinity.
The relationship between the $a_i, a_i', b_i, b_i',$ and $c_2$ and the new versions in higher degree is given by an invertible change of variables. By choosing $m$ large enough, it follows that, for each fixed codimension $j$, we have $A^j(\V_{r,d} \times_{\BSL_2} \V_{s,e})$ is generated by the $a_i, a_i', b_i, b_i',$ and $c_2$ and there are no relations among them.
\end{proof}

We now have a good understanding of the intersection theory of 
$\widehat{\B}^d_{2,g}$. However, our morphsim $\J^d_{2,g} \to \widehat{\B}^d_{2,g}$ factors through the open substack $\B^d_{2,g} \subset \widehat{\B}^d_{2,g}$ where the splitting type of $\E$ is not too unbalanced. Our next task is therefore to determine the relations on $\B^d_{2,g}$ that occur from bounding the splitting type of $\E$.

\subsection{Universal splitting loci} \label{usl}
Given integers $i, j$ with $i + j = e$, 
we define the \emph{universal splitting locus} $\Sigma_{i,j} \subset \V_{2,e}$ as follows. It is the substack whose points over a scheme $S$ are pairs $(P, E)$ where $P \to S$ is a $\pp^1$-bundle and $E$ is a rank $2$ vector bundle such that $E$ has locally constant splitting type $(i, j)$, see \cite[Section 3]{789}. As explained there, we have an equivalence of stacks 
\[\Sigma_{i,j} \cong \mathrm{B} \mathrm{Aut}(\O(i) \oplus \O(j)) \rtimes \SL_2.\]
As such, $\Sigma_{i,j}$ has codimension $j - 1 - 1$ inside $\V_{2,e}$. The closure of $\Sigma_{i,j}$ is the union of the $\Sigma_{i',j'}$ with $i' \leq i$.

These splitting loci are significant for us because
\[\B^d_{2,g} = \widehat{\B}^d_{2,g} \smallsetminus \overline{\Sigma}_{\lceil \frac{d+1}{2} \rceil, \lfloor \frac{d+1}{2} \rfloor - g - 2}.\]
Notice that $\overline{\Sigma}_{\lceil \frac{d+1}{2} \rceil, \lfloor \frac{d+1}{2} \rfloor - g - 2}$ has codimension $g+1$ if $d$ is odd and codimension $g+2$ when $d$ is even. In either case, since we assume throughout that $g \geq 2$, it follows that
\begin{equation} \label{pe} \Pic( \widehat{\B}^d_{2,g}) = \Pic(\B^d_{2,g}) \qquad \text{and} \qquad  \Pic(\widehat{\sB}^d_{2,g}) =  \Pic(\sB^d_{2,g}).
\end{equation}
However, we still need to understand the relations that are introduced in higher codimension.

The goal of this section is to prove the following closed formula for the fundamental class of the closure $[\overline{\Sigma}_{i,j}] \in A^{j - i - 1}(\V_{2,e})$ in terms of our generators $a_1, a_2, a_2'$ and $c_2$.

\begin{lem} \label{ff}
If $e$ is odd, then $j - i - 1$ is even, and $[\overline{\Sigma}_{i,j}] \in A^{j - i - 1}(\V_{2,e})$ is a non-zero multiple of
\[
\prod_{i < k <\frac{e}{2}}
(ka_1 - a_2')((e-k)a_1 - a_2') + (e-2k)^2(a_2 + k(e-k)c_2).\]
If $e$ is even, then $j - i - 1$ is odd, and $[\overline{\Sigma}_{i,j}] \in A^{j - i - 1}(\V_{2,e})$ is a non-zero multiple of
\[(ea_1 - 2a_2') \cdot \prod_{i < k < \frac{e}{2}}
(ka_1 - a_2')((e-k)a_1 - a_2') + (e-2k)^2(a_2 + k(e-k)c_2).
\]
\end{lem}

\begin{rem} A formula for $[\overline{\Sigma}_{i,j}]$ in terms of Chern classes of pushforward sheaves $\pi_* \E(m)$ is given in \cite[Lemma 5.1]{L}. In principle, one could use Grothendeick--Riemann--Roch to find the Chern classes of $\pi_* \E(m)$ in terms of our desired generators and plug that in. The generating function in \cite[Lemma 5.4]{Brd}
provides the conversion modulo the ideal generated by $c_2$. Following this route, one can derive formulas that agree with those in Lemma \ref{ff} modulo $c_2$. (In the eventual context where we apply this formula, it actually only matters what the answer is modulo $c_2$.)
However, the argument given below has the benefit of providing a conceptual reason for why the formulas factor and giving the more general result when $c_2 \neq 0$, which may be of use elsewhere.
\end{rem}

\begin{proof}
 Our basic strategy is to induct on the difference $j - i$.
The key step is to find the fundamental class of $\overline{\Sigma}_{i-1,j+1}$, viewed as a codimension $2$ class on $\overline{\Sigma}_{i,j}$ (or codimension $1$ class in the case $i = j$).
We will prove that $[\overline{\Sigma}_{i -1, j+1}]$ is equivalent to the pullback of a class in $A^2(\V_{2,e})$. Using the push-pull formula and applying induction then gives rise to the factoring formula in the statement.

To set up the argument, consider the restriction diagram
\begin{center}
\begin{tikzcd}
\iota'^*\E \arrow{d} \arrow{r} & \E \arrow{d} \\
P \arrow{d}[swap]{p} \arrow{r}{\iota'} & \P \arrow{d}{\pi} \\
\Sigma_{i,j} \arrow{r}{\iota} & \V_{2,e}.
\end{tikzcd}
\end{center}
First suppose $i \neq j$. As explained in \cite[Section 3]{789}, the stack $\Sigma_{i,j}$ comes equipped with two line bundles $\N_1$ and $\N_2$ (called the ``HN bundles") such that $\iota'^* \E$ sits in an exact sequence
\begin{equation} \label{filt} 0 \rightarrow (p^*\N_2)(j) \rightarrow  \iota'^* \E \rightarrow (p^*\N_1)(i) \rightarrow 0.
\end{equation}
By \cite[Lemma 3.1]{789}, we have $A^*(\Sigma_{i,j}) = \qq[n_1, n_2, c_2]$ where $n_k = c_1(\N_k)$.
Using \eqref{filt}, we find
\begin{align*}
\iota'^*c_1(\E) &= (p^*n_1 + iz) + (p^*n_2 + jz) = p^*(n_1 + n_2) + ez \\
\iota'^*c_2(\E) &= (p^*n_1 + iz)(p^*n_2 + jz) = p^*(n_1n_2 - ijc_2) + p^*(in_2 + jn_1)z.
\end{align*}
It follows that
\begin{equation} \label{pf}\iota^* a_1 = n_1 + n_2, \qquad \iota^* a_2 = n_1n_2 - ij c_2, \qquad \text{and} \qquad \iota^*a_2' = in_2 + jn_1.
\end{equation}

Since $A^*(\V_{2,e})$ is a domain \cite{Brd}, the fundamental class $[\overline{\Sigma}_{i,j}]$ is a non-zero divisor. By the push-pull formula, it follows that the Gysin pullback $\overline{\iota}^*: A^*(\V_{2,e}) \to A^*(\overline{\Sigma}_{i,j})$ is injective. 
Now, $A^2(\V_{2,e})$ has dimension $5$ with basis $a_1^2, a_1a_2', a_2',a_2,c_2$. Thus, $\dim A^{2}(\overline{\Sigma}_{i,j}) \geq 5$.
Meanwhile,
by excision, we have
\begin{equation} \label{e0} A^0(\overline{\Sigma}_{i-1,j+1}) \rightarrow A^2(\overline{\Sigma}_{i,j}) \rightarrow A^2(\Sigma_{i,j}) \rightarrow 0.
\end{equation}
By the discussion above, we know that $A^2(\Sigma_{i,j})$ has dimension $4$ with basis $n_1^2, n_1n_2, n_2^2, c_2$. It follows that $\overline{\Sigma}_{i-1,j+1} \in A^2(\overline{\Sigma}_{i,j})$ is non-zero and
\begin{equation} \label{iso} A^2(\V_{2,e}) = \langle a_1^2, a_1a_2', a_2',a_2,c_2 \rangle \xrightarrow{\overline{\iota}^*} A^2(\overline{\Sigma}_{i,j}) = \langle [\overline{\Sigma}_{i-1,j+1}], n_1^2,n_1n_2,n_2^2,c_2 \rangle
\end{equation}
is an isomorphism of $\qq$-vector spaces.
In particular, $[\overline{\Sigma}_{i-1,j+1}]$ lies in the image of $\overline{\iota}^*$. Now let $\alpha_{i-1,j+1}$ be a generator for the kernel of
\[A^2(\V_{2,e}) = \langle a_1^2, a_1a_2', a_2',a_2,c_2 \rangle \xrightarrow{\iota} A^2(\Sigma_{i,j}) = \langle n_1^2,n_1n_2,n_2^2,c_2 \rangle.\]
By \eqref{e0} and \eqref{iso}, we see that  $[\overline{\Sigma}_{i-1,j+1}] \in A^2(\overline{\Sigma}_{i,j})$ is a non-zero multiple of $\overline{\iota}^*\alpha_{i-1,j+1}$.

It remains to determine $\alpha_{i-1,j+1}$. From  \eqref{pf} we see that
\begin{align} 
\iota^*(ia_1 - a_2') &= (i - j)n_2 \label{n2} \\
\iota^*(ja_1 - a_2') &= (j - i)n_1 \label{n1} \\
\iota^*(a_2 + ijc_2) &= n_1n_2. \label{n1n2}
\end{align}
This helps us spot a generator of the kernel, namely any non-zero multiple of
\[\alpha_{i-1,j+1} = (ia_1 - a_2')(ja_1 - a_2') + (i - j)^2(a_2 + ijc_2).\]
By the push-pull formula, $[\overline{\Sigma}_{i-1,j+1}]$ is a non-zero multiple of $ \alpha_{i-1,j+1} [\overline{\Sigma}_{i,j}]$.
Inducting on the difference $j - i$ completes the proof in the case $e$ is odd.

If $e$ is even, we need a small variation on this argument to find the class of $\overline{\Sigma}_{i-1,i+1}$ inside $\overline{\Sigma}_{i,i} = \V_{2,e}$, since it is codimension $1$. Here, there is an exact sequence
\[A^0(\overline{\Sigma}_{i-1,i+1})\rightarrow A^1(\V_{2,e}) \to A^1(\Sigma_{i, i}) \rightarrow 0.\]
In this case, there is a rank $2$ HN bundle $\N$ with the property that $\iota'^*\E \cong (p^*\N)(i)$. However, if we think of $\N_1$ and $\N_2$ as the Chern roots of $\N$ and $A^*(\Sigma_{i,i})$ as the subring of symmetric polynomials in $n_1, n_2$, then
the calculations in \eqref{pf} are still valid.  In particular, we see that the kernel of $A^1(\V_{2,e}) \to A^1(\Sigma_{i,i})$ is generated by $ia_1 - a_2'$. Hence, $[\overline{\Sigma}_{i-1,i+1}]$ is a non-zero multiple of $ia_1 - a_2' = \frac{e}{2}a_1 - a_2'$.
\end{proof}

We make note of a specialization of these formulas that will be relevant in later sections. Although the universal formula on $\V_{2,e}$ is not a multiple of a perfect power, the specialization we wind up using is. The following is easy to check directly.
\begin{cor} \label{specialize}
If we substitute $c_2 = 0$ and $a_2 = \frac{1}{4}a_1^2$, then the formula for $[\overline{\Sigma}_{i,j}]$ becomes a non-zero multiple of $(ea_1 - 2a_2')^{j - i - 1}$.
\end{cor}

\subsection{The Chow ring of $\B^d_{2,g}$} \label{cb}
We are now ready to find the rational Chow ring of $\B^d_{2,g}$.
Let us define $\Sigma_{i,j}(\E) \subset \B^d_{2,g}$ to be the preimage of $\Sigma_{i,j}$ along the projection $\B^d_{2,g} \to \V_{2,d - g - 1}$.
We first note an important corollary of the proof of Lemma \ref{ff}.

\begin{cor} \label{pbcor}
(1) For any $i, j$ with $i + j = d - g - 1$, 
the pullback map 
\[A^*(\widehat{\B}^d_{2,g}) \to A^*(\Sigma_{i,j}(\E))\]
is surjective.

(2) If $\Sigma_{i,j}(\E) \subset \B^d_{2,g}$, then
\[A^*(\B^d_{2,g}) \to A^*(\Sigma_{i,j}(\E))\]
is also surjective.
\end{cor}
\begin{proof}
If $i = j$, then $\Sigma_{i,i}(\E)$ is open inside $\widehat{\B}^d_{2,g}$, so (1) is immediate. If $i \neq j$, then $A^*(\Sigma_{i,j})$ is generated by $n_1, n_2, c_2$. Moreover, $\Sigma_{i,j}(\E) \to \Sigma_{i,j}$ is a trivial $\gg_m$-gerbe (as $\V_{1,g+1}$ is a copy of $\mathrm{B}\gg_m \times \BSL_2$ by Remark \ref{bgm}).
Hence, $A^*(\Sigma_{i,j}(\E))$ is generated by $n_1, n_2, c_2, b_1$. It is easy to see that $c_2$ and $b_1$ are pulled back from $\widehat{\B}^d_{2,g}$.
By \eqref{n1} and \eqref{n2}, the classes $n_1$ and $n_2$ also lie in the image of the pullback map.

(2) follows from (1): If $\Sigma_{i,j}(\E) \subset \B^d_{2,g}$, then the pullback map $A^*(\widehat{\B}^d_{2,g}) \to A^*(\Sigma_{i,j}(\E))$ factors through $A^*(\B^d_{2,g})$. Thus, in order for (1) to be surjective, (2) must be surjective.
\end{proof}

With this, we can now completely determine the Chow ring of $\B^d_{2,g}$, which is the key takeaway from this section.

\begin{lem} \label{theb}
Let $S^d_g$ be the fundamental class of the closure of the 
$(\lceil \frac{d+1}{2} \rceil, \lfloor \frac{d+1}{2} \rfloor - g - 2)$
splitting locus for $\E$. Note that $S^d_g \in A^{g+1}(\widehat{\B}^d_{2,g})$ when $d$ is odd and 
$S^d_g \in A^{g+2}(\widehat{\B}^d_{2,g})$ when $d$ is even.
In either case,
\[A^*(\B^d_{2,g}) = \frac{A^*(\widehat{\B}^d_{2,g})}{\langle S^d_g\rangle } = \frac{\qq[a_1,a_2',a_2,b_1,c_2]}{\langle S^d_g\rangle }.\]
\end{lem}
\begin{proof}
By Lemma \ref{free}, we know that $A^*(\widehat{\B}_{2,g}^d) = \qq[a_1, a_2', a_2, b_1, c_2']$. 
The complement of $\B^d_{2,g} \subset \widehat{\B}^d_{2,g}$ is the union of splitting loci $\Sigma_{i,j}(\E)$ with $j - i \geq g+2$. Let $m = g+1$ if $d$ is odd and $g+2$ if $d$ is even, so it is the codimension of the complement.
Because the pullback map in Corollary \ref{pbcor}(1) is surjective, repeated use of excision and the push-pull formula shows that the image of
\[A^{*-m}\left(\bigcup_{j - i \geq g + 2} \Sigma_{i,j}(\E) \right) \rightarrow A^*(\B^d_{2,g})\]
is the ideal generated by the fundamental classes of the $\overline{\Sigma}_{i,j}(\E)$ with $j - i \geq g +2$. Considering the formulae in Lemma \ref{ff}, we see that each of these classes is a multiple of $S^d_g$.
\end{proof}

\subsection{Factoring the morphism} \label{stackdef}
We now have a good understanding of the Chow ring of $\B^d_{2,g}$. In order to use that to understand the Chow ring of $\J^d_{2,g}$, we need to understand the fibers of the morphism $\J^d_{2,g} \to \B^d_{2,g}$. The results of this section follow from \cite{ErmanWood} by taking the $\SL_2$ or $\PGL_2$ quotient of the stacks used there. Before stating the result, we include some motivation, which may also be helpful in understanding the perspective we take in later sections.

Informally, when we ask about the fibers of $\J^d_{2,g} \to \B^d_{2,g}$ we are asking: Given a rank $2$ vector bundle $E$ and line bundle $F$, what is the space of all pairs $(\alpha: C \to \pp^1, L)$ such that $E = \alpha_*L$ and $F = (\alpha_*\O_C/\O_{\pp^1})^\vee$? 
Given such a pair $(\alpha: C \to \pp^1, L)$, there is a surjective map $\alpha^*E = \alpha^*\alpha_*L \to L$ on $C$. This gives rise to a diagram
\begin{center}
\begin{tikzcd}
C \arrow{r}{\iota} \arrow{dr}[swap]{\alpha} & \pp E^\vee \arrow{d}{\gamma} \\
& \pp^1
\end{tikzcd}
\end{center}
where $\iota$ is an embedding.
The curve $C$ has relative degree $2$ over $\pp^1$, so it is the vanishing of a section of $\O_{\pp E^\vee}(2) \otimes \gamma^* A$ for some line bundle $A$ on $\pp^1$. To identify the line bundle $A$, consider the sequence
\[0 \rightarrow \O_{\pp E^\vee}(-2) \otimes \gamma^* A^\vee \rightarrow \O_{\pp E^\vee} \rightarrow \iota_* \O_C \rightarrow 0.\]
Applying $\gamma_*$ we obtain
\[0 \rightarrow \O_{\pp E^\vee} \rightarrow \alpha_*\O_C \rightarrow A^\vee \otimes R^1\gamma_* \O_{\pp E^\vee}(-2) = A^\vee \otimes \det E^\vee \rightarrow 0, \]
where the equality above follows from relative Serre duality (the relative dualizing sheaf is $\omega_{\gamma} = \gamma^*\det E \otimes \O_{\pp E^\vee}(-2)$ by the relative Euler sequence). It follows that $A^\vee \otimes \det E^\vee = F$, so $A = F \otimes \det E^\vee$. Thus, $C$ is defined by an element of the vector space
\[X_{E,F} := H^0(\pp E^\vee, \O_{\pp E^\vee}(2) \otimes \gamma^*(F\otimes \det E^\vee)) = H^0(\pp^1, \Sym^2 E \otimes F \otimes \det E^\vee). \]
Thus, the fiber of $\J^d_{2,g} \to \B^d_{2,g}$ over the point $(E, F)$ should be the open subset of $X_{E, F}$ corresponding to equations whose vanishing locus defines a smooth curve.

Note that the line bundle $F \otimes \det E^\vee$ has degree $g + 1 - (d - g - 1) = 2g + 2 - d$. (This quantity is denoted $k$ in \cite[Section 5]{ErmanWood}; the line bundle $A$ above corresponds to the line bundle Erman and Wood call $L$.)
Suppose $E$ has splitting type $(e_1, e_2)$ with $e_1 \leq e_2$.
Then $\Sym^2 E \otimes F \otimes \det E^\vee$ has splitting type
\[(2e_1 + 2g + 2 - d, e_1 + e_2 + 2g + 2 - d, 2e_2 + 2g + 2 - d).\]
Because $e_1 \geq \frac{d}{2} - g - 1$ (see Section \ref{themor}), all the degrees above are nonnegative.

By slight abuse of notation, let $\pi: \P \to \B^d_{2,g}$ and $\E, \F$ denote the restriction of the universal bundles to $\B^d_{2,g} \subset \widehat{\B}^d_{2,g}$.
Because the splitting type of $\E$ on the fibers of $\P \to \B^d_{2,g}$ satisfies $e_1 \geq \frac{d}{2} - g - 1$, cohomology and base change implies that $\pi_* \Sym^2 \E \otimes \F \otimes \det \E^\vee$ is locally free on $\B^d_{2,g}$ and its fibers are identified with the vector spaces $X_{E, F}$. Let 
\[\mathcal{X} = \pi_* \Sym^2 \E \otimes \F \otimes \det \E^\vee\]
be the total space of this vector bundle on $\B^d_{2,g}$. Let $\mathscr{X}$ denote the analogous vector bundle on $\sB^d_{2,g}$.

The above discussion suggests (and more or less suggests a proof) that the morphism $\J^d_{2,g} \to \B^d_{2,g}$ factors through an open inclusion in the total space of the vector bundle $\mathcal{X}$ on $\B^d_{2,g}$. This is what Erman and Wood showed. They worked with a fixed $\pp^1$, in other words with the pullback of our situation to $\B^{d,\dagger}_{2,g}$. The pullback of $\X$ along $\B^{d,\dagger}_{2,g} \to \B^d_{2,g}$ is the stack Erman and Wood denote $\overline{\sJ}^{2,g,d}_{\mathrm{bd}}$. Theorem 5.1 of \cite{ErmanWood} says that $\J^{d,\dagger}_{2,g}$ (which is denoted $\sJ^{2,g,d}$ in \cite{ErmanWood}) is an open substack of $\overline{\sJ}^{2,g,d}_{\mathrm{bd}}$. The theorem below therefore follows by taking the $\SL_2$ quotient or $\PGL_2$ quotient.

\begin{thm}[Erman--Wood] \label{os}
The stack $\J^d_{2,g}$ is an open substack of $\mathcal{X}$. Similarly, $\sJ^d_{2,g}$ is an open substack of $\mathscr{X}$.
\end{thm}

\begin{rem}
The identification of $\overline{\sJ}^{2,g,d}_{\mathrm{bd}}$ with the pullback of $\X$ along $\B^{d,\dagger}_{2,g} \to \B^d_{2,g}$ can be found in the proof of \cite[Theorm 5.1]{ErmanWood}. Erman and Wood first construct the total space of a vector bundle, denoted $\mathbb{V}(\F^\vee)$ there, on a stack they call $\mathscr{U}^{2,d}_k$. In our notation, $\mathscr{U}^{2,d}_k$ is the open substack of $\V^\dagger_{2,d - g + 1}$ where the splitting type satisfies $e_1 \geq \frac{d}{2} - g - 1$. Then they take a $\gg_m$ quotient of their $\mathbb{V}(\F^\vee)$. The result is a vector bundle on $\mathscr{U}^{2,d}_k \times \mathrm{B} \gg_m$; this stack is what we have called $\B^{d,\dagger}_{2,g}$ (see Remark \ref{bgm}).
\end{rem}

Since $\mathcal{X} \to \B^d_{2,g}$ is a vector bundle, their Chow rings are isomorphic. We have already found the Chow ring of $\B^d_{2,g}$. In order to determine the Chow ring of $\J^d_{2,g}$, we must determine the relations that arise from excising the complement of $\J^d_{2,g} \subset \mathcal{X}$. The complement is a discriminant locus corresponding to the equations in each $X_{E,F}$ that define singular curves in $\pp E^\vee$. To get a better understanding of this complement, we use the machinery of relative bundles of principal parts.

\section{Principal parts} \label{partsec}

Here we define and state some basic properties of relative bundles of principal parts that we shall need in the next section. We shall only need first order principal parts. For more on bundles of principal parts, we refer the reader to \cite[Sections 7 and 11]{EH}; or, for further variations, see \cite[Section 3]{part2}.

Let $b: Y \to Z$ be a smooth proper surjective morphism.
Let $\Delta_{Y/Z}\subset Y \times_Z Y$ be the relative diagonal and let $p_1, p_2: Y\times_ZY \to Y$ be the two projection maps.
Let $\W$ be a vector bundle on $Y$ and let $\mathcal{I}_{\Delta_{Y/Z}}$ denote the ideal sheaf of the diagonal in $Y \times_Z Y$. The bundle of relative first order principal parts $P^1_{Y/Z}(\W)$ is defined as
\[
P^1_{Y/Z}(\W)=p_{2*}(p_1^*\W\otimes \mathcal{O}_{Y \times_Z Y}/\mathcal{I}_{\Delta_{Y/Z}}^{2}).
\]
The fiber of $P^1_{Y/Z}(\W)$ at a point $y \in Y$ corresponds to the restriction of $\W$ to a first order neighborhood of $y$.
We shall need the following properties of $P^1_{Y/Z}(\W)$:
\begin{prop}[Theorem 11.2 of \cite{EH}]\label{parts}
With notation as above,
\begin{enumerate}
    \item There is an isomorphism $b^*b_*\W\xrightarrow{\sim} p_{2*}p_1^*\W$.
    \item The quotient map $p_1^*\W\rightarrow p_1^*\W\otimes \mathcal{O}_{Y \times_Z Y}/\mathcal{I}_{\Delta_{Y/Z}}^{2}$ pushes forward to a map
    \[
    b^*b_*\W\cong p_{2*}p_1^*\W\rightarrow P^{1}_{Y/Z}(\W),
    \]
    which, fiber by fiber, associates to a global section $f$ of $\W$ a section $f'$ whose value at $z\in Z$ is the restriction of $f$ to a first order neighborhood of $z$ in the fiber $b^{-1}b(z)$.
    \item \label{filtration} The order of vanishing of sections gives rise to a filtration
    \[
    0\rightarrow \W\otimes \Omega_{Y/Z} \rightarrow P^1_{Y/Z}(\W)\rightarrow \W \rightarrow 0.
    \]
\end{enumerate}
\end{prop}

Below are some criteria to help us determine when the evaluation map from part (2) is surjective. Proofs of the following lemmas can be found in \cite[Section 3]{part2}.
\begin{lem}[Lemma 3.3 of \cite{part2}] \label{va}
Suppose $\W$ is a relatively very ample line bundle on $Y$ over $Z$. Then the evaluation map
\[
b^*b_*\W\rightarrow P^1_{Y/Z}(\W)
\]
is surjective. 
\end{lem}

\begin{lem}[Lemma 3.5 of \cite{part2}]  \label{cl}
Suppose $\E$ is a vector bundle on a $\pp^1$-bundle $\pi: \P \to B$ and let $\gamma: \pp \E^\vee \to \P$ be the projectivization. Suppose $\W = (\gamma^*A) \otimes \O_{\pp E^\vee}(m)$ for some $m \geq 1$ and vector bundle $A$ on $\P$.
If $R^1\pi_*[\gamma_*\W \otimes \O_{\P}(-2)] = 0$, then the evaluation map
\[(\pi \circ \gamma)^*(\pi \circ \gamma)_* \W \to P^1_{\pp \E^\vee/B}(\W)\]
is surjective.
\end{lem}

Finally, suppose we have a tower $X \xrightarrow{a} Y \xrightarrow{b} Z$ of smooth proper surjective morphisms. We have the relative cotangent sequence
\[0 \rightarrow a^*\Omega_{Y/Z} \rightarrow \Omega_{X/Z} \rightarrow \Omega_{X/Y} \rightarrow 0.\]
Tensoring the left map above with $\W$ and combining with Proposition \ref{parts}(3), we have an inclusion 
\[\W \otimes a^*\Omega_{Y/Z} \hookrightarrow \W \otimes \Omega_{X/Z} \hookrightarrow P^1_{X/Z}(\W).\]
The quotient gives rise to a surjection
$P^1_{X/Z}(\W) \to P^1_{X/Y}(\W)$.

\section{Excising the discriminant} \label{ex}
Recall (see Theorem \ref{os}) that $\J^d_{2,g}$ is an open substack of the vector bundle $\mathcal{X}$ on $\B^d_{2,g}$. By cohomology and base change, the fiber of $\mathcal{X}$ over a point corresponding to $(E, F) \in \B^d_{2,g}$ is the vector space
\begin{equation} \label{h0} X_{E,F} := H^0(\pp E^\vee, (F \otimes \det E^\vee) \otimes \O_{\pp E^\vee}(2)) = H^0(\pp^1, F \otimes \det E^\vee \otimes \Sym^2 E),
\end{equation}
Let $\Delta \subset \X$ be the complement of $\J^d_{2,g}$. The fiber of $\Delta$ over the point $(E, F) \in \B^d_{2,g}$ is the subvariety of equations in \eqref{h0} whose vanishing defines a singular curve, union the zero section ($\Delta$ is a cone).
Informally,
\[\Delta = \{(E, F, f) : f \in X_{E,F} \text{ and there exists $p$ such that } \dim T_pV(f) = 2\}.\]
The universal singular point makes a choice of point $p$:
\[\widetilde{\Delta} = \{(E, F, f, p) : p \in \pp E^\vee, \ f \in X_{E,F} \text{ and } \dim T_pV(f) = 2 \}.\]
By excision, there is a right exact sequence pictured in the top row below. Meanwhile, we have a proper, surjective morphism $\widetilde{\Delta} \to \Delta$. Since we work with rational coefficients, this induces the surjection in the left vertical map below. 
\begin{equation} 
\begin{tikzcd}
A^{*-1}(\Delta) \arrow{r} & A^*(\mathcal{X}) \arrow{r} & A^*(\J^d_{2,g}) \arrow{r} &0 \\
   A^{*-1}(\widetilde{\Delta}) \arrow[two heads]{u} & A^*(\B^d_{2,g}) \arrow{u}[swap]{\cong}
\end{tikzcd}
\end{equation}
In particular, the above diagram implies that
\begin{equation} \label{aj} A^*(\J^d_{2,g}) = A^*(\X)/\mathrm{Im}(A^{*-1}(\tilde{\Delta}) \to A^*(\X)).
\end{equation}

\subsection{The universal singular point} \label{usp}
Our first step is to construct the universal singular point and study its fibers over the universal Hirzebruch surface. We can detect when an equation $f \in X_{E, F}$ defines a singular curve using bundles of principal parts. 
Given an equation $f \in X_{E, F}$, one obtains a section of an associated principal parts bundle on $\pp E^\vee$ which tracks the values and first derivatives of $f$ at each point in $\pp E^\vee$.

Let $\gamma: \pp \E^\vee \to \P$ be the universal Hirzebruch surface. Let $\W := \gamma^*(\F \otimes \det \E^\vee) \otimes \O_{\pp \E^\vee}(2)$ be the line bundle on $\pp \E^\vee$ whose sections on fibers over $\B^d_{2,g}$ are the spaces $X_{E, F}$. 
On $\pp \E^\vee$, we consider 
\begin{equation} \label{ppev} \gamma^*\pi^* \X = \gamma^*\pi^*(\pi_*\gamma_*\W) \to P^1_{\pp \E^\vee/\B^d_{2,g}}(\W). 
\end{equation}
The principal parts bundle above has rank $3$ and is filtered by three line bundles
\begin{equation} \label{three} \W, \qquad \Omega_{\gamma} \otimes \W, \qquad \text{and} \qquad \Omega_{\pi} \otimes \W,
\end{equation}
where $\Omega_{\gamma}$ and $\Omega_{\pi}$ denote the relative contangent bundles of $\gamma$ and $\pi$ respectively.
These three line bundles correspond to the value of a section, its derivative along a fiber of $\pp \E^\vee \to \P$, and its derivative in a complementary direction within the fiber of $\pp \E^\vee \to \B^d_{2,g}$, respectively.
In other words, at a point $p \in \pp \E^\vee$, the map \eqref{ppev} sends the space of equations $X_{E,F}$ to their value and these first order derivatives at $p$.
Thus, by construction, $\tilde{\Delta}$ is the preimage of the zero section of \eqref{ppev}.

If the evaluation map is surjective, then
the preimage of the zero section is just the total space of the kernel subbundle. We shall soon prove that \eqref{ppev} is surjective when $d$ is odd.
However, it turns out that, when $d$ is even, \eqref{ppev} is not surjective over all of $\pp \E^\vee$, making this case more complicated.
In either case though, we have the following diagram
\begin{equation} \label{td}
\begin{tikzcd}
\tilde{\Delta} \arrow{r}{\jmath} \arrow{dr}[swap]{\rho''} & \gamma^*\pi^*\X \arrow{rr}{\sigma'} \arrow{d}{\rho'} & & \X \arrow{d}{\rho} \\
&    \pp \E^\vee \arrow{r}[swap]{\gamma} & \P \arrow{r}[swap]{\pi} & \B^d_{2,g}.
\end{tikzcd}
\end{equation}
Because of \eqref{aj}, our goal is to determine the image of $A^{*-1}(\tilde{\Delta}) \to A^*(\X)$. 
We explain the case when $d$ is odd first. Then we explain the case when $d$ is even, which requires some extra arguments.

\subsection{Notation}
Here, we set up some notation and calculations that will be needed for both cases $d$ even and odd.
Let $z = c_1(\O_{\P}(1))$ and let $\zeta = c_1(\O_{\pp \E^\vee}(1))$. Recall that we write $c_1(\E) = a_1 + a_1'z, c_2(\E) = a_2 + a_2'z$ and $c_1(\F) = b_1 + b_1' z$. We also set $c_2 = -z^2$.
Here, and in what follows, we will omit various pullbacks along $\pi$ and $\gamma$ applied to classes in order to declutter notation; it should always be clear where a class is pulled back from.
With this notation,
\begin{equation} \label{A} c_1(\W) = c_1(\F) - c_1(\E) + 2c_1(\O_{\pp \E^\vee}(1)) = (b_1 - a_1) + (b_1' - a_1')z + 2\zeta.
\end{equation}
We compute the top Chern class of the principal parts bundle using the splitting principle.
\begin{lem} \label{c3}
We have
 \begin{align*} 
 c_3(P^1_{\pp \E^\vee/\B^d_{2,g}}(\W)) &= (8g+4)b_1 \cdot z\zeta +
 4(b_1^2 + (g^2 + g)c_2) \cdot \zeta \\
 &\quad +
 [(g + 1)(a_1^2 - 4a_2) + \langle b_1,c_2 \rangle ] \cdot z + \langle b_1, c_2 \rangle.
\end{align*}
\end{lem} 

\begin{proof}
The principal parts bundle is filtered by the three line bundles in \eqref{three}.
Using the relative Euler sequence for the projective bundles $\gamma: \pp \E^\vee \to \P$ and $\pi: \P \to \B^d_{2,g}$, we find
\begin{equation} \label{rt} \Omega_\gamma = \det \E \otimes \O_{\pp \E^\vee}(-2) \qquad \text{and} \qquad \Omega_{\pi} = \O_{\P}(-2).
\end{equation}
Therefore,
\begin{equation} \label{Oe} c_1(\Omega_{\gamma}) = (a_1 + a_1' z) - 2\zeta \qquad \text{and} \qquad c_1(\Omega_{\pi}) = -2z.
\end{equation}
Combining \eqref{A} and \eqref{rt}, the product of the first Chern classes of the three line bundles in \eqref{three} is
\begin{equation} \label{prod}
((b_1 - a_1) + (b_1' - a_1')z + 2\zeta)(b_1 + b_1'z)((b_1 - a_1) + (b_1' - a_1' - 2)z + 2\zeta).
\end{equation}
To obtain the form stated in the lemma, we use the projective bundle theorem to rewrite this product as a linear combination of
$1, z, \zeta, z\zeta$ with coefficients pulled back from $A^*(\B^d_{2,g})$. Specifically, we apply the rule $\zeta^2 - c_1(\E)\zeta + c_2(\E) = 0$ to remove powers of $\zeta^2$; after that we use $z^2 + c_2 = 0$ to remove powers of $z^2$.
\end{proof}

\subsection{Odd degree}
The thing that makes odd degree simpler is that the principal parts evaluation map \eqref{ppev} is surjective, and hence $\tilde{\Delta} \to \pp \E^\vee$ is the total space of a vector bundle.
\begin{lem} \label{os}
If $d$ is odd, then \eqref{ppev} is surjective.
\end{lem}
\begin{proof}
We apply Lemma \ref{cl} which says the evluation map is surjective so long as the higher pushforward $R^1\pi_*[\gamma_*\W \otimes \O_{\P}(-2)] = 0$.
 By cohomology and base change, it suffices to check that the restriction of 
$\gamma_*\W \otimes \O_{\P}(-2)$ to each fiber of $\pi$ has vanishing $H^1$.

Suppose $\E$ has splitting type $(e_1, e_2)$ with $e_1 \leq e_2$ on some fiber of $\pi$.
Then, on that fiber, the vector bundle $\gamma_* \W \otimes \O_{\P}(-2) = \F \otimes \det \E^\vee \otimes \Sym^2 \E \otimes \O_{\P}(-2)$ is isomorphic to
\begin{equation} \label{st} \O(g + 1) \otimes \O(-d + g + 1) \otimes (\O(2e_1) \oplus \O(e_1 + e_2) \oplus \O(2e_2)) \otimes \O(-2). 
\end{equation}
The lowest degree summand above is $2e_1 - d + 2g$. By the definition of $\B^d_{2,g}$, we always have $e_1 \geq \frac{d}{2} - g - 1$. However, since $d$ is odd, this implies $e_1 \geq \frac{d+1}{2} - g - 1$. In turn, this says $2e_1 - d + 2g \geq -1$. In particular, all summands in \eqref{st} have degree $\geq -1$ and the desired $H^1$ vanishes on every fiber of $\pi$.
\end{proof}

\begin{lem} \label{oddfinal}
If $d$ is odd, then
\[A^*(\J^d_{2,g}) = \frac{A^*(\B_{2,g}^d)}{\langle b_1, c_2, a_1^2 - 4a_2 \rangle} \cong \frac{\qq[a_1,a_2']}{\langle ((d - g - 1) a_1 - 2a_2')^{g+1}\rangle}.\]   
\end{lem}

\begin{proof}
Since \eqref{ppev} is surjective, $\tilde{\Delta} \to \pp \E^\vee$ is the total space of a vector bundle. 
Since it has codimension three, the fundamental class of $\tilde{\Delta}$ inside the total space of $\gamma^*\pi^*\X$ is given by $\rho'^*c_3(P_{\pp \E^\vee/\B^d_{2,g}}^1(\W))$, which was calculated in Lemma \ref{c3}. Using the diagram \eqref{td} and the push-pull formula, we compute
\begin{equation} \label{Dclass} [\Delta] = \sigma'_* \jmath_*[\tilde{\Delta}] = \sigma'_* \rho'^* c_3(P_{\pp \E^\vee/B^d_{2,g}}^1(\W)) = \rho^*\pi_*\gamma_*c_3(P_{\pp \E^\vee/B^d_{2,g}}^1(\W)) = \rho^* (8g + 4)b_1. 
\end{equation}
More generally, we can compute all classes pushed forward from $\tilde{\Delta}$ as follows.
Since $\tilde{\Delta}$ is the total space of a vector bundle, the pullback map $A^*(\pp \E^\vee) \to A^*(\tilde{\Delta})$ is an isomorphism. Moreover, $A^*(\pp \E^\vee)$ is generated over $A^*(\B^d_{2,g})$ as a module by $1, z, \zeta, z\zeta$.
Thus, by \cite[Lemma 2.1]{part2}, the image of $A^{*-1}(\tilde{\Delta}) \to A^*(\X)$ is the (pullback from $\B^d_{2,g}$ of) the ideal generated by the four classes
\[\pi_*\gamma_* (c_3(P_{\pp \E^\vee/B^d_{2,g}}^1(\W)) z^i \zeta^j) \qquad \qquad 0 \leq i, j \leq 1.
\]
When we compute these pushforwards using the projective bundle theorem, it's clear that the ideal these classes generate is the same as the ideal generated by the coefficients of $1, z, \zeta, z\zeta$ in the formula for $c_3(P^1_{\pp \E^\vee/\B^d_{2,g}}(\W))$ given in Lemma \ref{c3}. It follows that image of $A^{*-1}(\tilde{\Delta}) \to A^*(\X)$ is the ideal generated by $b_1, c_2, a_1^2 - 4a_2$.

To obtain the last equality in the lemma, we use Lemma \ref{theb}. Specifically, we eliminate $b_1, c_2$ and set $a_2 = \frac{1}{4}a_1^2$ in $S^d_g$. By Corollary \ref{specialize}, this is a multiple of $(ea_1 - 2a_2')^{g+1}$. 
\end{proof}
\begin{rem}
Note that $b_1$ and $c_2$ are pulled back from the moduli space $\H_{2,g}$ of hyperelliptic curves. Since $A^*(\H_{2,g}) = \qq$, we could have known a priori that these classes must vanish.
\end{rem}

\subsection{Even degree}
Examining the proof of Lemma \ref{os}, we see that when $d$ is even, \eqref{ppev} might drop rank over the locus where $\E$ has splitting type $(\frac{d}{2} - g - 1, \frac{d}{2})$ on fibers of $\pi$. This is the most extreme splitting type allowed and corresponds to the splitting type of $\alpha^*\O_{\pp^1}(\frac{d}{2})$.

We now examine the picture over this splitting locus in more detail. Let $i := \frac{d}{2} - g - 1$ and $j := \frac{d}{2}$.
By slight abuse of notation, we denote by $\Sigma_{i,j}$ and $\N_1, \N_2$
the pullbacks of the universal splitting locus and HN bundles on it along the map $\B^d_{2,g} \to \V_{2,d - g - 1}$ (what was previously denoted $\Sigma_{i,j}(\E)$). We also continue to write $p: P \to \Sigma_{i,j}$ for the restriction of the universal $\pp^1$-bundle.

 With this notation, we have the diagram below, where both of the squares on the right are fiber squares:
\begin{equation} \label{bigd}
\begin{tikzcd}
\mathcal{Z} := \pp(p^*\N_1(i))^\vee \arrow{r} \arrow[bend left = 20]{rr}{\tau} \arrow{dr}[swap]{\sim} & \pp(\iota'^*\E)^\vee  \arrow{r} \arrow{d} & \pp \E^\vee \arrow{d}{\gamma} \\
& P \arrow{d}[swap]{p} \arrow{r}{\iota'} & \P \arrow{d}{\pi} \\
& \Sigma_{i,j} \arrow{r}{\iota} & \B^d_{2,g}
\end{tikzcd}
\end{equation}
The map $\pp(p^*\N_1(i))^\vee \to \pp(\iota'^*\E)^\vee$ is induced by the dual of the right-hand map in \eqref{filt}.
In words, $\Z$ corresponds to the directrix of the universal Hirzebruch surface over the most extreme splitting type allowed. For future reference, it is helpful to note that
\begin{equation} \label{tz} \tau^*\zeta = c_1(\O_{\pp(p^*\N_1(i))^\vee}(1)) = c_1(p^*\N_1(i)) = p^*n_1 + iz \in A^1(\Z) \cong A^1(P).
\end{equation}

\begin{lem} \label{es} If $d$ is even, then \eqref{ppev} is surjective on the complement of $\Z$.
\end{lem}
\begin{proof}
Arguing as in the proof of Lemma \ref{os}, we see that \eqref{ppev} is surjective unless the smallest summand in \eqref{st} satisfies $2e_1 - d + 2g \leq -2$. 
 By the definition of $\B^d_{2,g}$, we always have $e_1 \geq \frac{d}{2} - g - 1$. Hence, if \eqref{ppev} fails to be surjective, then $e_1 = \frac{d}{2} - g - 1  = i$. However, as we show next, the map doesn't drop rank along all of $\gamma^{-1}\pi^{-1}(\Sigma_{i,j}) = \pp (\iota'^*\E)^\vee$; it just drops rank along $\Z$.

We analyze the situation on a fixed fiber of $\pi$ over $\Sigma_{i,j}$.
Let $E = \O_{\pp^1}(i) \oplus \O_{\pp^1}(j)$ and let $W = \O_{\pp^1}(d - 2g - 2) \otimes \O_{\pp E^\vee}(2)$. By Lemma \ref{va}, it suffices to show that the linear system associated to $W$ defines an embedding away from the directrix of $\pp E^\vee$.
Note that $W = (\O_{\pp^1}(i) \otimes \O_{\pp E^\vee}(1))^{\otimes 2}$,
so it suffices to see that $\O_{\pp^1}(i) \otimes \O_{\pp E^\vee}(1)$ induces an embedding away from the directrix of $\pp E^\vee$. Indeed, we now recognize $\O_{\pp^1}(i) \otimes \O_{\pp E^\vee}(1) = \O_{\pp(\O \oplus \O(j-i))^\vee}(1)$ as the map which sends the Hirzebruch surface $\pp E^\vee$ to the cone over a rational normal curve of degree $j-i = g+1$ (contracting the directrix to the cone point).
\end{proof}

Recall the notation in the diagram \eqref{td}.
We will study $\widetilde{\Delta}$ by breaking it into two pieces:
\[\widetilde{\Delta} = \rho''^{-1}(\pp \E^\vee \smallsetminus \Z) \cup \rho''^{-1}(\Z).\]
Lemma \ref{es} shows that $\rho''^{-1}(\pp \E^\vee \smallsetminus \Z) \to \pp \E^\vee \smallsetminus \Z$ is the total space of a vector bundle.
Our next step is to identify $\rho''^{-1}(\Z) \to \Z$ as the total space of a vector bundle on $\Z$. In other words, even though $\tilde{\Delta} \to \pp \E^\vee$ is not the total space of a vector bundle, it is a union of two pieces, each of which is a vector bundle over a corresponding piece of $\pp \E^\vee$.

In the lemma below, we make use of the relative principal parts bundle $P^1_{\pp \E^\vee/\P}(\W)$, which is a rank $2$ quotient of $P^1_{\pp \E^\vee/\B^d_{2,g}}(\W)$.
\begin{lem}
The map
\[\gamma^*\pi^*\X|_{\Z} \to P^1_{\pp \E^\vee/\P}(\W)|_{\Z}\]
is surjective and $\rho''^{-1}(\Z) \subset \rho'^{-1}(\Z) = \gamma^*\pi^*\X|_{\Z}$ is the total space of the kernel.
\end{lem}
\begin{proof}
Since $\W$ has relative degree $0$ on $\Z$, the rank of \eqref{ppev} is at most $2$ along $\Z$. To see it is exactly $2$, we 
consider the following diagram
\begin{center}
\begin{tikzcd}
\gamma^*\pi^* \pi_* \gamma_* \W \arrow{r} \arrow{d} & P^1_{\pp \E^\vee/\B^d_{2,g}}(\W) \arrow{d} \\
\gamma^*\gamma_* \W \arrow{r} & P^1_{\pp \E^\vee/\P}(\W).
\end{tikzcd}
\end{center}
Since $\W$ has relative degree $2$ on the fibers of $\gamma$, the bottom row is surjective by Lemma \ref{va}.
Meanwhile, the left vertical arrow is surjective because $\gamma_*\W$ is globally generated on the fibers of $\pi$ (all summands are degree $\geq 0$).

We have now shown that \eqref{ppev} has constant rank $2$ along $\Z$. Thus, $\rho''^{-1}(\Z)$ is a vector subbundle of $\gamma^*\pi^*\X|_{\Z}$ of corank $2$. Moreover, it lies in the kernel of the surjective map $\gamma^*\pi^*\X \to P^1_{\pp \E^\vee/\P}(\W)$ to a rank $2$ bundle, so $\rho''^{-1}(\Z)$ must indeed be the kernel.
\end{proof}

It follows that the fundamental class of $\rho''^{-1}(\Z) \subset \rho'^{-1}(\Z)$ is 
\begin{equation} \label{rclass} [\rho''^{-1}(\Z)] = \rho'^*c_2(P^1_{\pp \E^\vee/\P}(\W)|_{\Z}) \in A^2(\rho'^{-1}(\Z)).
\end{equation}
We calculate this Chern class below.

\begin{lem} \label{rlem}
Let $e = d - g - 1$. We have
\[c_2(P^1_{\pp \E^\vee/\P}(\W)|_{\Z}) = (ea_1 - 2a_2')z + \langle b_1 \rangle \in A^2(\Z).\]
\end{lem}
\begin{proof}
The bundle $P^1_{\pp \E^\vee/\P}(\W)$ is 
filtered by $\W$ and $\Omega_{\gamma} \otimes \W$. 
Therefore the top Chern class is given by the product 
\[c_2(P^1_{\pp \E^\vee/\P}(\W)) = c_1(\W) c_1(\Omega_{\gamma} \otimes \W).\]
Our task is to compute the pullback of this class along $\tau: \mathcal{Z} \to \pp \E^\vee$. By \eqref{tz}, we have $\tau^*\zeta = p^*n_1 + iz$. (Note that here, by slight abuse of notation, we are writing $z$ for the pullback of $z$ along $P \to \P$; this is reasonable because it's still the class of $\O(1)$ on this projective bundle.)
Using \eqref{n1}, we write $n_1 = \frac{1}{j - i}\iota^*(j a_1 - a_2') = \frac{1}{g+1}\iota^*(ja_1 - a_2')$. Recall that $a_1' = d - g - 1$ and $b_1' = g+1$.
Applying $\tau^*$ to \eqref{A}, we find
\begin{align*}
\tau^* c_1(\W) &= (b_1 - a_1) + (b_1' - a_1')z + \frac{2}{g+1}(ja_1 - a_2') + 2iz \\
&= b_1 - a_1 + \frac{2}{g+1}(ja_1 - a_2') + (b_1' - a_1' + 2i)z \\
&= b_1 + \frac{1}{g+1}\left(ea_1 - 2a_2'\right),
\end{align*}
where all instances of $a_1, a_2',b_1$ are understood to be pulledback from $\B^d_{2,g}$.
Notice that the coefficient of $z$ above vanishes; this is because $\W$ has relative degree $0$ along $\Z$ (which is also the reason that the rank of the evaluation map drops rank along $\Z$!)
On the other hand, combining \eqref{A} and \eqref{Oe}, we have $c_1(\Omega_{\gamma} \otimes \W) = b_1 + b_1'z = b_1 + (g+1)z$. Multiplying the two first Chern classes gives the formula in the statement.
\end{proof}

With these extra ingredients, we can now calculate the Chow ring of $\J^d_{2,g}$ when $d$ is even.

\begin{lem} \label{even}
If $d$ is even, then
\[A^*(\J^d_{2,g}) = \frac{A^*(\B_{2,g}^d)}{\langle b_1, c_2, a_1^2 - 4a_2, [\Sigma_{i,j}] \cdot (ea_1 - 2a_2') \rangle} \cong \frac{\qq[a_1,a_2']}{\langle ((d - g - 1)a_1 - 2a_2')^{g+1}\rangle}\]
\end{lem}
\begin{proof}
We again define $\tilde{\Delta} \subset \gamma^*\pi^*(\pi_*\gamma_*\W)$ to be the preimage of the zero section of \eqref{ppev} so that we have the diagram \eqref{td}.
However, because \eqref{ppev} fails to be surjective over all of $\pp \E^\vee$, it is no longer true that $\rho''^*: A^*(\pp\E^\vee) \to A^*(\widetilde{\Delta})$ is surjective.
To proceed, we stratify $\tilde{\Delta}$ into two pieces:
\[\tilde{\Delta} = \rho''^{-1}(\Z) \cup \rho''^{-1}(\pp \E^\vee \smallsetminus \Z),\]
where each piece is a vector bundle over its image in $\pp \E^\vee$.
Thus, we have maps described below where the rows are right exact:
\begin{center}
\begin{tikzcd}
A^*(\rho''^{-1}(\Z)) \arrow{r} & A^*(\widetilde{\Delta}) \arrow{r}  & A^*(\rho''^{-1}(\pp \E^\vee \smallsetminus \Z)) \arrow{r} & 0 \\
A^*(\Z) \arrow[two heads]{u} \arrow{r} & A^*(\pp \E^\vee)  \arrow{u}{\rho''^*} \arrow{r} & A^*(\pp \E^\vee \smallsetminus \Z) \arrow[two heads]{u} \arrow{r} & 0.
\end{tikzcd}
\end{center}
However, the left square in the diagram above does \emph{not} commute! The composition along the bottom and then up differs from going up and then across by multiplication by an ``excess intersection factor" associated to the fact that the fiber dimension of $\rho''$ jumps over $\Z$.
(The right square does commute.)
Although the middle vertical map is not surjective, both of the outer ones are.
It follows that $A^*(\widetilde{\Delta})$ is generated by the image of $\rho''^*$ together with the image of the composition
\[A^*(\Z) \cong A^*(\rho''^{-1}(\Z)) \to A^*(\tilde{\Delta}).\]

We are ultimately interested in finding the image of $A^{*-1}(\widetilde{\Delta}) \to A^*(\X)$. 
Note that the fundamental class of $\tilde{\Delta}$ inside $\gamma^*\pi^*\X$ is still given by 
$\rho'^*c_3(P^1_{\pp \E^\vee/\B^d_{2,g}}(\W))$ since it has codimension $3$ and is defined by the vanishing of a section of that rank $3$ principal parts bundle.
In particular, arguing as in Lemma \ref{oddfinal} shows that the image of 
\[A^{*-1}(\pp \E^\vee) \xrightarrow{\rho''^*} A^{*-1}(\widetilde{\Delta}) \to A^*(\X) \cong A^*(\B^d_{2,g})\]
is the ideal $\langle b_1, c_2, a_1^2 - 4a_2 \rangle$.
Next we find the contribution from 
\[A^{*-g-1}(\rho''^{-1}(\Z)) \to A^{*-1}(\tilde{\Delta}) \to A^*(\X) \cong A^*(\B^d_{2,g}).\]
Consider the following diagram:
\begin{center}
\begin{tikzcd}
\rho''^{-1}(\Z) \arrow{r}{\jmath} \arrow{dr}[swap]{\rho''} & \rho'^{-1}(\Z)  \cong p^*\iota^*\X \arrow{r}{p'} \arrow{d}{\rho'} & \iota^*\X \arrow{r}{\iota'} \arrow{d} & \X \arrow{d}{\rho} \\
& \Z \cong P \arrow{r}[swap]{p} \arrow{r} & \Sigma_{i,j} \arrow{r}[swap]{\iota} & \B^d_{2,g}.
\end{tikzcd}
\end{center}
If we ignore the rightmost column above, the left part of the diagram is a ``trapezoid diagram" based over $\Sigma_{i,j}$.
Since $\mathcal{Z} \to \Sigma_{i,j}$ is a $\pp^1$-bundle, $A^*(\mathcal{Z})$ is generated over $A^*(\Sigma_{i,j})$ by $1$ and $z$.
Thus, as in \cite[Lemma 2.1]{part2}, the image of 
\[p'_*\circ \jmath_*: A^{*-1}(\rho''^{-1}(\Z)) \to A^*(\iota^*\X) \cong A^*(\Sigma_{i,j})\] 
is the ideal generated by $p'_*\jmath_* 1$ and $p'_*\jmath_* \rho''^*z$.
By Corollary \ref{pbcor}(2), $\iota^*: A^*(\B^d_{2,g}) \to A^*(\Sigma_{i,j})$ is surjective.
It follows that the image of 
\[\iota'_*\circ p'_*\circ \jmath_*:
A^{*-g-1}(\rho''^{-1}(\Z)) \to A^*(\X) \cong A^*(\B^d_{2,g})
\] is the ideal generated by $\iota'_*p'_*\jmath_* 1$ and $\iota'_*p'_*\jmath_* \rho''^*z$.

In order to compute $\jmath_* 1$, we combine \eqref{rclass}
 and Lemma \ref{rlem} to get
 \[\jmath_* 1 = (ea_1 - 2a_2')z + \langle b_1 \rangle \in A^2(\rho'^{-1}(\Z)). \]
 Recall that above, $a_1, a_2'$ and $b_1$ are pulled back from $\B^d_{2,g}$.
Thus, we have
\[\iota'_*p_*'\jmath_*1 = \iota'_* \iota'^*(ea_1 - 2a_2') + \langle b_1 \rangle = 
[\Sigma_{i,j}] \cdot (ea_1 - 2a_2') + \langle b_1 \rangle \in A^*(\X) \cong A^*(\B^d_{2,g}). \]
Meanwhile, using the fact that $z^2 = -c_2$, we have
\[\iota'_*p'_*\jmath_*\rho''^{*}z =  \iota'_*p'_*\jmath_*\jmath^*\rho'^*z = \iota'_*p'_*(\jmath_*1 \cdot \rho'^*z) \in \langle c_2, b_1 \rangle \subset A^*(\X) \cong A^*(\B^d_{2,g}).\]

 In conclusion, the image of $A^{*-1}(\widetilde{\Delta}) \to A^*(\X) \cong A^*(\B^d_{2,g})$ is the ideal
 \[\langle [\Sigma_{i,j}] \cdot (ea_1 - 2a_2'), b_1, c_2, a_1^2 - 4a_2 \rangle. \]
Equation \eqref{aj} thus proves the first equality in the statement of the lemma.
For the second, Corollary \ref{specialize} tells us that $[\Sigma_{i,j}]$ is a non-zero multiple of $(ea_1 - 2a_2')^g$ modulo $\langle c_2, a_1^2  - 4a_2 \rangle$.
We then apply Lemma \ref{theb} and note that, also by Corollary \ref{specialize}, the class $S^d_g$ already lies in the ideal above.
\end{proof}

\section{The integral Picard group} \label{picsec}
\subsection{Proof of Theorem \ref{pthm}}
The computation of the integral Picard group of $\J^d_{2,g}$ follows readily from the computations we have already done.
Since $\X$ is a vector bundle over $\B^d_{2,g}$, we have
\[\Pic(\X) = \Pic(\B^d_{2,g}) = \Pic(\widehat{\B}_{2,g}) = \zz a_1 \oplus \zz a_2' \oplus \zz b_1,\]
where the second equality is \eqref{pe} and the third equality is found in Lemma \ref{pin}. By excision,
\[\Pic(\J^d_{2,g}) = \Pic(\X \smallsetminus \Delta) = \Pic(\X)/\langle [\Delta] \rangle  = \frac{\zz a_1 \oplus \zz a_2' \oplus \zz b_1}{\langle (8g + 4)b_1 \rangle}.\]
Note that the computation of $[\Delta]$ in \eqref{Dclass} holds in odd and even degree: Even though \eqref{ppev} is not always surjective, $\tilde{\Delta}$ still has codimension $3$ and thus has class $c_3(P^1_{\pp \E^\vee/\B^d_{2,g}}(\W))$. 
The same calculation can be performed base changing everything along $\B^{d, \dagger}_{2,g} \to \B^d_{2,g}$. Since this map induces an isomorphism on Picard groups, it follows that $\Pic(\J^{d,\dagger}_{2,g}) = \Pic(\J^d_{2,g})$. 

\subsection{Proof of Theorem \ref{picthm}}
To obtain the Picard group of $\sJ^d_{2,g}$, we bootstrap our calculations for the $\SL_2$ quotient along with Lemma \ref{pin}.
Recall Theorem \ref{os} that $\sJ^d_{2,g}$ is an open substack of the vector bundle $\mathscr{X}$ on $\sB^d_{2,g}$. Let $\mathscr{D}$ denote the complement $\mathscr{X} \smallsetminus \sJ^d_{2,g}$. By construction, we have a fiber diagram
\begin{center}
\begin{tikzcd}
\Delta \arrow{r} \arrow{d} & \mathscr{D} \arrow{d} \\
\X \arrow{r}{\xi} \arrow{d} & \mathscr{X} \arrow{d} \\
\B^d_{2,g} \arrow{r} & \sB^d_{2,g}.
\end{tikzcd}
\end{center}
The horizontal maps are $\mu_2$-gerbes.
By Lemma \ref{pin} and \eqref{pe}, the pullback map $\xi^*$ includes $\Pic(\mathscr{X}) = \Pic(\sB^d_{2,g})$
as a full rank subgroup of
$\Pic(\X) = \Pic(\B^d_{2,g})$. Moreover, we know $\xi^*[\mathscr{D}] = [\Delta] = (8g + 4)b_1$. If $g$ is odd, then $b_1$ lies in the subgroup $\Pic(\mathscr{X}) \subset \Pic(\X)$ and it follows that
\[\Pic(\sJ^d_{2,g}) = \Pic(\mathscr{X} \smallsetminus \mathscr{D}) = \Pic(\mathscr{X})/\langle \mathscr{D} \rangle = \zz \oplus \zz \oplus \zz/(8g + 4).\]
However, when $g$ is even, $b_1$ does not lie in the subgroup $\Pic(\mathscr{X})$; rather, $2b_1$ is the smallest multiple of $b_1$ that lies in the subgroup. In this case, $[\mathscr{D}] = (4g + 2)(2b_1)$ is $4g+2$ times a generator.
It follows that
\[\Pic(\sJ^d_{2,g}) = \Pic(\mathscr{X} \smallsetminus \mathscr{D}) = \Pic(\mathscr{X})/\langle \mathscr{D} \rangle = \zz \oplus \zz \oplus \zz/(4g + 2),\]
completing the proof of Theorem \ref{picthm}.

\subsection{Construction of line bundles} \label{clb}
In this section, we construct explicit line bundles representing the generators of $\Pic(\sJ^d_{2,g})$.
Over $\sJ^d_{2,g}$, we have a universal diagram
\begin{center}
\begin{tikzcd}
\sC \arrow{rr}{\alpha} \arrow{dr}[swap]{f} && \sP \arrow{dl}{\pi} \\
& \sJ_{2,g}^d
\end{tikzcd}
\end{center}
where $\sC$ is the universal curve, $\sP$ is a family of genus $0$ curves, and $\alpha: \sC \to \sP$ is the universal degree $2$ cover.
The pushforward $\alpha_* \sL$ is a rank $2$ vector bundle of relative degree $d - g - 1$ on fibers of $\pi$. 
Let us define
\[\mathscr{A}:= \begin{cases}  \pi_*((\det \alpha_*\sL)^{\otimes 2} \otimes \omega_{\pi}^{\otimes d - g - 1}) \qquad &\text{if $d - g - 1$ is odd} \\
\pi_*((\det \alpha_*\sL) \otimes \omega_{\pi}^{\otimes (d - g - 1)/2})
&\text{if $d - g - 1$ is even.}
\end{cases} \]
By construction, we are pushing forward a line bundle that has relative degree $0$ on the fibers of $\pi$, so $\mathscr{A}$ is a line bundle on $\sJ^d_{2,g}$ by cohomology and base change. Notice that
\begin{equation} \label{c1A} c_1(\mathscr{A}) = \begin{cases} 2a_1 &\text{if $d - g - 1$ is odd.} \\
a_1 &\text{if $d - g - 1$ is even.}
\end{cases}
\end{equation}
We can perform a similar construction with $\alpha_*\O$. We define
\[\mathscr{B}:= \begin{cases}  \pi_*((\det \alpha_*\O)^{\otimes 2} \otimes \omega_{\pi}^{\otimes - g - 1}) \qquad &\text{if $g$ is even} \\
\pi_*((\det \alpha_*\O) \otimes \omega_{\pi}^{\otimes (- g - 1)/2})
&\text{if $g$ is odd.}
\end{cases} \]
Then we have
\begin{equation} \label{c1B} c_1(\mathscr{B}) = \begin{cases} -2b_1 &\text{if $g$ is even.} \\
-b_1 &\text{if $g$ is odd.}
\end{cases}
\end{equation}
This line bundle is the pull back of the generator of $\Pic(\mathscr{H}_{2,g})$.
We can also combine $\alpha_*\sL$ and $\alpha_*\O$ as follows.
Note that $\det(\alpha_*\sL) \otimes \det(\alpha_*\O)$ has degree $d - 2g - 2$.
 We define
\[\mathscr{N} := \begin{cases} \det(\alpha_*\sL) \otimes \det(\alpha_*\O) \otimes \omega_\pi^{\otimes (d/2 - g - 1)} &\text{if $d$ is even} \\
 \det(\alpha_*\sL)^{\otimes 2}  \otimes \det(\alpha_*\O)^{\otimes 2} \otimes \omega_\pi^{\otimes (d - 2g - 2)} &\text{if $d$ is odd,}
\end{cases}
\]
which has
\begin{equation} \label{c1N} c_1(\mathscr{N}) = \begin{cases} a_1 - b_1 &\text{if $d$ is even} \\
2a_1 - 2b_1 &\text{if $d$ is odd.}
\end{cases}
\end{equation}

Finally, we restrict the line bundle $\Lambda(0, 1)$ on $\sJ^d_g$ defined in \cite[p. 3]{MeloViviani} to $\sJ^d_{2,g}$. From its definition and \cite[Equation 2.4]{MeloViviani}, we have
\[c_1(\Lambda(0,1)) = c_1(f_*\sL) - c_1(R^1f_*\sL) = c_1(\pi_*(\alpha_*\sL)) - c_1(R^1\pi_*(\alpha_*\sL)).\]
Recall that $\alpha_*\sL$ is rank $2$ and degree $d - g - 1$.
Applying Grothendieck--Riemann--Roch for pushforward along $\pi$ exactly as in \cite[Example 3.4]{Pic} we find
\begin{equation} \label{c1Lam}
c_1(\Lambda(0, 1)) = (d - g)a_1 - a_2'.
\end{equation}
Combining equations \eqref{c1A}, \eqref{c1B}, \eqref{c1N}, and \eqref{c1Lam} along with Lemma \ref{pin}, one readily checks that the line bundles $\mathscr{A}, \mathscr{B}, \mathscr{N}$, and $\Lambda(0, 1)$ generate $\Pic(\sJ^d_{2,g})$. (Depending on the parity of $d$ and $g$, different subsets of three of these suffice to generate.)

\subsection{Proof of Corollary \ref{bc}}
Since $\nu_d: \sJ^d_{2,g} \to J^d_{2,g}$ is a $\gg_m$-gerbe, their Picard groups are related by the following exact sequence (see \cite[Equation 6.1]{MeloViviani}):
\begin{equation} \label{MVeq} 0 \rightarrow \Pic(J^d_{2,g}) \rightarrow \Pic(\sJ^d_{2,g}) \xrightarrow{\mathrm{res}} \Pic(\mathrm{B} \gg_m) \cong \zz \xrightarrow{\mathrm{obs}} \mathrm{Br}(J^d_{2,g}).
\end{equation}
Above, the map $\mathrm{res}$ sends a line bundle to the character of its restriction to a fiber of the $\gg_m$-gerbe, and $\mathrm{obs}$ sends $1$ to the class of $\nu_d$ in the Brauer group of $J^d_{2,g}$. In particular, the order of $\nu_d$ in the Brauer group of $J^d_{2,g}$ is the index of the image of $\mathrm{res}$.

Recall that $\gg_m$ acts on the fibers of $\mathscr{L}$ by scaling. Thus, using the definitions in the previous subsection, we find
\begin{equation} \label{AB} 
\mathrm{res}(\mathscr{A}) = \begin{cases} 4 &\text{if $d - g - 1$ is odd} \\ 2 & \text{if $d - g - 1$ is even} \end{cases} \qquad \mathrm{res}(\mathscr{B}) = 0 \qquad \mathrm{res}(\mathscr{N}) =\begin{cases} 2 & \text{if $d$ is even} \\ 4 & \text{if $d$ is odd.} \end{cases}
\end{equation}
Moreover, by \cite[Lemma 6.2]{MeloViviani}, we have $\mathrm{res}(\Lambda(0, 1)) = d - g + 1$. Because these four line bundles generate $\Pic(\sJ^d_{2,g})$ this allows us to determine the image of $\mathrm{res}$. If $d - g +1$ is odd, then the image of $\mathrm{res}$ contains $4$ and an odd number, so it is surjective. Meanwhile, if $d - g + 1$ is even, then $d - g - 1$ is also even, and we see that the image of $\mathrm{res}$ is $2\zz$.

\section{Tautological classes and relations} \label{taut}

In this section, we provide formulas for the $\kappa_{i,j}$ in terms of our generators $a_1$ and $a_2'$. This allows us to translate Lemmas \ref{oddfinal}  and \ref{even} into the presentation given in Theorem \ref{intro}. It also allows us to determine all relations among the $\kappa_{i,j}$ restricted to the hyperelliptic locus, given in Corollary \ref{relscor}. When $d = g - 1$, we also calculate the class of the theta divisor. In the last section, we calculate the Chow ring of the rigidification.

\subsection{Formulas for $\kappa_{i,j}$}
Since there is an isomorphism of rational Chow rings $A^*(\sJ^d_{2,g}) \cong A^*(\J^d_{2,g})$, it suffices to determine the relations when pulled back to $A^*(\J^d_{2,g})$.
Let $f: \C \to \J^d_{2,g}$ be the universal curve. Let $\pi: \P \to \J^d_{2,g}$ be the universal genus $0$ curve.
Since we are working with the $\SL_2$ quotient, $\P$ comes equipped with a relative degree $1$ line bundle $\O_{\P}(1)$.
We then have a diagram
\begin{center}
\begin{tikzcd}
\C \arrow{rr}{\alpha} \arrow{dr}[swap]{f} && \P \arrow{dl}{\pi} \\
& \J_{2,g}^d
\end{tikzcd}
\end{center}
where $\alpha$ is a degree $2$ cover.
Let $\L$ be the universal degree $d$ line bundle on $\C$. Pushing forward along $\alpha$ gives rise to a rank $2$ vector bundle $\E := \alpha_* \L$ on $\P$. Notice that, by slight abuse of notation, we are now using $\P, \E, \pi$ and $\gamma$ for the pullbacks of the universal families along $\J^d_{2,g} \to \B^d_{2,g}$. Let us also write $\W$ for the pullback of the relative degree $2$ line bundle on $\pp \E^\vee$ defined in Section \ref{usp}. We write $z = c_1(\O_{\P}(1))$ and $\zeta = c_1(\O_{\pp \E^\vee}(1))$.

The natural map $\alpha^* \E = \alpha^*\alpha_*\L \to \L$ defines an embedding $\iota: \C \hookrightarrow \pp \E^\vee$.
The universal curve is the vanishing of a section of $\W$. 
As such, the fundamental class of $\C$ inside $\pp \E^\vee$ is
\begin{equation} \label{classC} [\C] = c_1(\W) = (b_1 - a_1) + (b_1' - a_1')z + 2\zeta \in A^1(\pp \E^\vee).
\end{equation}
Here, as before, we are omitting the pullbacks along $\pi$ and $\gamma$ on $a_i, b_i$ to declutter notation.
By construction, the universal line bundle on $\C$ is the restriction of $\O_{\pp \E^\vee}(1)$, so
\begin{equation} \label{c1l} c_1(\L) = \iota^*\zeta \in A^1(\C).
\end{equation}
We can find the relative dualizing sheaf using adjunction:
\[\omega_f = \iota^*(\omega_{\pi \circ \gamma} \otimes \W) = \iota^*(\Omega_{\pi} \otimes \Omega_{\gamma} \otimes \W).\]
Adding together the formulas in \eqref{A} and \eqref{Oe}, we see that $c_1(\omega_f)$ is the restriction of
\[c_1(\Omega_\gamma) + c_1(\Omega_{\pi}) + c_1(\W) = b_1 + (b_1' - 2)z.\]
Working rationally, we know that $b_1 = 0 \in A^*(\J^d_{2,g})$, so
\begin{equation} \label{c1of} c_1(\omega_f) = (g - 1)\cdot \iota^*z \in A^1(\C).
\end{equation}

Using \eqref{c1l} and \eqref{c1of}, and then applying \eqref{classC},
we obtain
\begin{align} 
\kappa_{i,j} &= f_*(c_1(\omega_f)^{i+1} \cdot c_1(\L)^j) = f_*(\iota^*[(g - 1)^{i+1}z^{i+1}\cdot \zeta^{j}]) \notag \\
&= \pi_*\gamma_* \iota_*(\iota^*[(g - 1)^{i+1}z^{i+1}\cdot \zeta^{j}]) = \pi_*\gamma_*([\C] \cdot (g-1)^{i+1} z^{i+1} \cdot \zeta^j) \notag \\
&= \pi_*\gamma_*((-a_1 + (2g + 2 - d)z + 2\zeta) \cdot  (g-1)^{i+1}z^{i+1} \cdot \zeta^j
). \label{last}
\end{align}
In the last line we have again used that $b_1 = 0$ to simplify the expression.
Recalling that $z^2 = -\pi^*c_2 = 0 \in A^2(\P)$, it follows that many of the twisted kappa classes vanish.
\begin{lem} \label{kvan}
We have $\kappa_{i,j} = 0$ for all $i \geq 1, j \geq 0$.
\end{lem}

Next we study the $\kappa_{0,j}$.
\begin{lem} \label{k0j}
We have 
\[\kappa_{0,j} = \frac{1}{2^{j-1}}(g - 1)a_1^j = \frac{1}{(2g - 2)^{j-1}} (\kappa_{0,1})^j.\]
\end{lem}
\begin{proof}
Let $h_j = \pi_*(\gamma_*((2\zeta - a_1)\zeta^j) \cdot z)$. Recall that $z^2 = -c_2 = 0$. Hence, by \eqref{last}, we have $\kappa_{0,j} = (g - 1)h_j$. By the projective bundle formula for $\pp \E^\vee$ over $\P$ we have
\begin{equation} \label{pbf} \zeta^2 - (a_1 + a_1'z)\zeta + (a_2 + a_2'z) = 0 \in A^*(\pp \E^\vee).
\end{equation}
Above, the classes $a_1, a_2, a_2'$ are pulled back from $\J^d_{2,g}$.
Substituting $a_2 = \frac{1}{4}a_1^2$, multiplying by $4$ and completing the square, we find that 
\begin{equation} \label{4} (2\zeta - a_1)^2 = (4a_1'\zeta - 4a_2')z \in \langle z \rangle.
\end{equation}
Thus, again using that $z^2 = 0$, we have
\begin{align*} 0 = \pi_*(\gamma_*((2\zeta - a_1)(2\zeta - a_1)\zeta^j) \cdot z) = 2h_{j+1} - a_1 h_j.
\end{align*}
It is easy to see that $h_0 = 2$. It then follows by induction that $h_j = \frac{1}{2^{j-1}} a_1^j$, which yields the claimed formula.
\end{proof}

A similar approach allows us to recursively determine $\kappa_{-1,j}$.
\begin{lem} \label{km1j}
For all $j \geq 1$, we have
\begin{align} 
\kappa_{-1,j} &= \frac{1}{2^{j-1}}a_1^{j-2}(jda_1 + (j^2-j)((d - g -1)a_1 -2a_2') ) \label{as} \\
&= \frac{1}{(2g-2)^{j-1}} \cdot(\kappa_{0,1})^{j-2}\cdot \left( (g-1)(j^2 - j) \cdot \kappa_{-1,2} - d(j^2 - 2j) \cdot \kappa_{0,1}\right) \label{ks}
\end{align}
\end{lem}
\begin{proof} 
By \eqref{last}, we have
\[\kappa_{-1,j} = (2g + 2 - d)\pi_*\gamma_*(z \cdot \zeta^j) + \pi_*\gamma_*((2\zeta - a_1)\zeta^j).\]
Let $f_j = \pi_*\gamma_*(\zeta^{j+1} \cdot z)$ and let $\ell_j = \pi_*\gamma_*((2\zeta - a_1)\zeta^{j+1})$, so that 
\begin{equation} \label{sothat} \kappa_{-1,j} = (2g +2 - d)f_{j-1} + \ell_{j-1}.
\end{equation}
Note that $f_0 = 1$.
Using notation from the previous lemma, we have
\[\frac{1}{2^{j-1}} a_1^j = h_j = 2 f_j - a_1 f_{j-1}.\]
This formula recursively determines $f_j$; one readily checks that the recursion is solved by
\begin{equation} \label{fj} f_j = \frac{j+1}{2^{j}}a_1^j.
\end{equation}
Meanwhile, using \eqref{4}, we have
\[(4a_1'\zeta - 4a_2')z \zeta^{j} = (2\zeta - a_1)^2\zeta^{j} = 2(2\zeta - a_1)\zeta^{j+1} - a_1(2\zeta - a_1)\zeta^{j}.\]
Thus, applying $\pi_*\gamma_*$, we find
\begin{equation} \label{diff} 4a_1' f_{j} - 4a_2' f_{j-1} = 2\ell_{j} - a_1 \ell_{j-1}. 
\end{equation}
Hence, applying \eqref{sothat} followed by \eqref{fj} and \eqref{diff}, we obtain the following recursion:
\begin{align*} 2\kappa_{-1,j+1} - a_1 \kappa_{-1,j} &= (2g + 2 - d)(2f_j - a_1f_{j-1}) + 2\ell_j - a_1\ell_{j-1} \\
&=(2g + 2 - d)\cdot \frac{1}{2^{j-1}}a_1^j + 4a_1' \cdot \frac{j+1}{2^j} a_1^j - 4a_2' \cdot \frac{j}{2^{j-1}}a_1^{j-1}. \\
\intertext{Recalling that $a_1' = d - g - 1$ and collecting terms, we simplify to}
&= \frac{1}{2^{j-1}}a_1^{j-1}\left((2g + 2 - d + 2(d - g - 1)(j+1))a_1 - 4j a_2' \right) \\
&= \frac{1}{2^{j-1}}a_1^{j-1}\left(da_1 + 2j((d - g - 1)a_1 - 2a_2') \right).
\end{align*}
This recursively determines all $\kappa_{-1,j}$.
For the base case, we compute directly from \eqref{last} and \eqref{pbf} that 
\[\kappa_{-1,1} = 2g + 2 - d + 2\pi_*\gamma_*\zeta^2 = 2g + 2 - d + 2a_1' = d.\]
The claimed formula in \eqref{as} is then verified using induction.

Setting $j = 2$ in \eqref{as} we find
\[\kappa_{-1,2} = (2d - g - 1)a_1 - 2a_2'.\]
By Lemma \ref{k0j}, we have $\kappa_{0,1}= (g-1)a_1$. Therefore, we have
\begin{align*} (j^2 - j)\kappa_{-1,2} - \frac{d(j^2 - 2j)}{(g-1)} \cdot \kappa_{0,1} &= (j^2 - j)(d + (d - g - 1)a_1 - 2a_2') - d(j^2 - 2j)a_1 \\
&=jda_1 + (j^2 - j)((d - g - 1)a_1 - 2a_2').
\end{align*}
Collecting the factors of $\frac{1}{g - 1}$ now proves the form in
\eqref{ks}.
\end{proof}

\subsection{Proofs of Theorem \ref{intro} and Corollary \ref{relscor}}

By Lemmas \ref{k0j} and \ref{km1j}, we have
\[\kappa_{0,1} = (g - 1)a_1 \qquad \text{and} \qquad \kappa_{-1,2} = (2d - g - 1)a_1 - 2a_2'.\]
Thus, we can solve
\[a_1 = \frac{1}{(g - 1)}\kappa_{0,1} \qquad \text{and} \qquad a_2' = \frac{1}{2} \left(\frac{2d - g - 1}{g - 1} \cdot \kappa_{0,1} - \kappa_{-1,2}\right).\]
In particular,
\[(d - g - 1)a_1 - 2a_2' 
= \kappa_{-1,2} - \frac{d}{g - 1}\kappa_{0,1}.\]
Thus, when written in terms of the generators $\kappa_{0,1}$ and $\kappa_{-1,2}$, the presentations of $A^*(\J^d_{2,g})$ in Lemmas \ref{oddfinal}
 and \ref{even} become the ring presented in Theorem \ref{intro}.

As for Corollary \ref{relscor}, the
first three families of relations in the statement follow from Lemmas \ref{kvan}, \ref{k0j}, and \ref{km1j} respectively.
The last relation and the fact that there are no more follow from Theorem \ref{intro}.

\subsection{The theta divisor and splitting loci} \label{theta}
Recall that when $d = g - 1$, there is a canonical divisor called the \emph{theta divisor} on $\sJ^{g-1}_g$ corresponding to those line bundles that have a non-zero global section. By Riemann--Roch, since $d = g - 1$, having a global section is the same as having non-zero $H^1$. Thus, by cohomology and base change, we can describe the theta divisor as the support of the sheaf $R^1f_* \sL$.
A standard calculation with the Grothendieck--Riemann--Roch formula expresses the class of the theta divisor in terms of the twisted kappa classes.

\begin{lem} \label{tl}
We have $\Theta = \frac{1}{2}\kappa_{0,1} - \frac{1}{2} \kappa_{-1,2} - \frac{1}{12}\kappa_{1,0} \in A^1(\sJ^{g-1}_g)$.
\end{lem}
\begin{proof}
When $d = g - 1$, a general line bundle has no global sections, so $f_*\sL = 0$. The theta divisor is the support of $R^1f_*\sL$, so has fundamental class $c_1(R^1f_*\sL)$.
By Grothendieck--Riemann--Roch, we have
\begin{align*} -c_1(R^1f_*\sL) &= \left[ f_*(\mathrm{Ch}(\sL) \mathrm{Td}(\omega_f^\vee)) \right]_1 \\
&= \left[f_* \left(1 + c_1(\sL) + \frac{c_1(\sL)^2}{2} + \ldots \right) \left(1 - \frac{c_1(\omega_f)}{2} + \frac{c_1(\omega_f)^2}{12} + \ldots \right)\right]_1 \\
&= f_*\left( \frac{c_1(\omega_f)^2}{12} - \frac{c_1(\sL) c_1(\omega_f)}{2} + \frac{c_1(\sL)^2}{2} \right) = \frac{1}{12} \kappa_{1,0} - \frac{1}{2}\kappa_{0,1} + \frac{1}{2}\kappa_{-1,2},
\end{align*}
from which the result follows.
\end{proof}
The formula given in Lemma \ref{tl} holds on the whole universal Picard stack. When we restrict to the hyperelliptic locus $\sJ^{g-1}_{2,g} \subset \sJ^{g-1}_g$ we have $\kappa_{1,0} = 0$ and so $\Theta$ is a multiple of $\kappa_{0,1} - \kappa_{-1,2}$. This is the class whose $(g+1)$st power generates the relations in Theorem \ref{intro}.

We also give a direct conceptual argument in terms of splitting loci.
When $d = g - 1$, the general line bundle has splitting type $(-1, -1)$, so the theta divisor is the same thing as the  splitting locus for splitting type $(-2, 0)$. We saw in Corollary \ref{specialize} that, once we set $c_2 = 0$ and $a_2 = \frac{1}{4}a_1^2$, 
the universal formulas for splitting loci become perfect powers of the class of the codimension $1$ splitting locus.
Thus, the class of the codimension $g+1$ splitting locus must be a multiple of the $(g+1)$st power of the class of theta.

\subsection{The Chow ring of the rigidification} \label{rigid}
In this last section, we study the relationship between the Chow ring of the stack $\sJ^d_{2,g}$ and of its rigidification $J^d_{2,g}$. 
The former is a $\gg_m$-gerbe over the latter. 

We first recall some general facts about $\gg_m$-gerbes. Suppose $\mathscr{X} \to X$ is a $\gg_m$-gerbe. 
There is a restriction map $\mathrm{res}: \Pic(\mathscr{X}) \to \Pic(\mathrm{B} \gg_m) = \zz$ that restricts a line bundle to a fiber of $\mathscr{X} \to X$ (as in \eqref{MVeq}). A line bundle $\mathscr{M}$ on $\mathscr{X}$ is called
$\emph{$n$-twisted}$ if $\mathrm{res}(\mathscr{M}) = n$ (see e.g. \cite[Chapter 12.3]{Olsson}).
We claim that if $\mathscr{M}$ is an $n$-twisted line bundle, then the total space of $\mathscr{M}$ minus the zero section is a $\mu_n$-gerbe over $X$. Indeed, the $\gg_m$ stabilizer at a point on $\mathscr{X}$ acts transitively on the fiber of $\mathscr{M} \smallsetminus 0$ with stabilizer $\mu_n$. Alternatively, let $\mathscr{X}_n \to X$ represent $n$ times the class of $\mathscr{X} \to X$. There is a natural map $\mathscr{X} \to \mathscr{X}_n$. Fiberwise over $X$ this map looks like the map $\mathrm{B}\gg_m \to \mathrm{B}\gg_m$ induced by $t \mapsto t^n$. Hence, $\mathscr{X} \to \mathscr{X}_n$ is a $\mu_n$-gerbe and pullback along it induces an equivalence of $1$-twisted line bundles on $\mathscr{X}_n$ with $n$-twisted line bundles on $\mathscr{X}$.
By \cite[Lemma 3.1.1.8]{lib}, $\mathscr{X}$ posesses an $n$-twisted line bundle if and only if $\mathscr{X}_n$ is trivial. In this case, we have the following diagram
\begin{center}
\begin{tikzcd}
\mathscr{M} \arrow{d} \smallsetminus 0 \arrow{r} & X \times \mathrm{pt} \arrow{d} \arrow{r} & \mathrm{pt} \arrow{d} \\
\mathscr{X} \arrow{dr} \arrow{r} & \mathscr{X}_n \cong X \times \mathrm{B} \gg_m \arrow{d} \arrow{r} & \mathrm{B} \gg_m \\
& X
\end{tikzcd}
\end{center}
 where the squares are fibered squares. This shows that $\mathscr{M} \smallsetminus 0 \to X$ is a $\mu_n$-gerbe.

\begin{lem} \label{gc}
We have $A^*(J^d_{2,g}) \cong \qq[u]/(u^{g+1})$, and the pullback map $A^*(J^d_{2,g}) \to A^*(\sJ^d_{2,g})$ is the inclusion defined by $u \mapsto d \kappa_{0,1} - (g - 1)\kappa_{-1,2}$. 
\end{lem}
\begin{proof}
Let $\mathscr{A}$ be the line bundle defined in Section \ref{clb}. By \eqref{AB}, we have that $\mathscr{A}$ is an $n$-twisted line bundle for $n = 4$ or $2$ depending on the parity of $d - g$. In particular, $\mathscr{A} \smallsetminus 0 \to J^d_{2,g}$ is a $\mu_n$-gerbe, so 
we have an isomorphism of rational Chow rings $A^*(J^d_{2,g}) \cong A^*(\mathscr{A} \smallsetminus 0)$. The latter is isomorphic to $A^*(\sJ^d_{2,g})/\langle c_1(\mathscr{A})\rangle = A^*(\sJ^d_{2,g})/\langle a_1 \rangle$.   
The first claim now follows from Lemmas \ref{oddfinal} and \ref{even}.

For the second claim, note that the composition
$A^*(J^d_{2,g}) \to A^*(\sJ^d_{2,g}) \to A^*(\mathscr{A} \smallsetminus 0)$ is an isomorphism on rational Chow rings. Hence, the first map must be an injection. Up to scale, there is a unique codimension $1$ class whose $(g+1)$st power vanishes; thus, up to rescaling the generator $u$, the inclusion $A^*(J^d_{2,g}) \to A^*(\sJ^d_{2,g})$ is defined by $u \mapsto d \kappa_{0,1} - (g - 1)\kappa_{-1,2}$. 
\end{proof}

\begin{rem}
When $d = 0$, an argument using the Beauville--Deninger--Murre weight decomposition for Chow groups of abelian schemes shows that $(\kappa_{-1,2})^{g+1} = 0 \in A^*(J^0_g)$. 
In this decomposition, the weight $w$ classes are those such that pullback along the multiplication by $n$ map $J^0_g \to J^0_g$ acts by $n^w$ for all integers $n$. One computes that 
$\kappa_{i,j}$ has weight $j$ and weights are multiplicative. By \cite[Theorem 2.19]{DM}, all classes of weight larger than $2g$ vanish. See also \cite[Proposition 4.2]{BMSY}.
\end{rem}

The codimension $1$ class $d \kappa_{0,1} - (g - 1)\kappa_{-1,2}$ also arises in the results \cite{MeloViviani,EW} on the Picard groups of $J^d_g$ and $\sJ^d_g$ over all of $\sM_g$. Let $\lambda$ denote the pullback of the first Chern class of the Hodge bundle on $\sM_g$.
There, it was shown that $\Pic(J^d_g) \to \Pic(\sJ^d_g)$ is the subgroup generated by $\lambda$ and $\frac{1}{\gcd(2g - 2, g+d - 1)}(d \kappa_{0,1} - (g - 1)\kappa_{-1,2})$. (We note that \cite{EW} defines $\kappa_{i,j}$ using powers of the relative tangent bundle instead of the relative cotangent bundle, so their $\kappa_{i,j}$ differs from ours by $(-1)^i$.)
There is a commutative diagram
\begin{equation}
\begin{tikzcd}
\sJ^d_{2,g} \arrow{d} \arrow{r} & \sJ^d_{g} \arrow{d} \\
 J^d_{2,g} \arrow{r} & J^d_g
\end{tikzcd}
\end{equation}
where the horizontal arrows are closed embeddings and the vertical arrows are $\gg_m$-gerbes. It follows that the pullback on rigidifications $A^*(J^d_g) \to A^*(J^d_{2,g})$ is also surjective.

\bibliographystyle{amsplain}
\bibliography{refs}

\end{document}